\documentclass[aop]{imsart}

\usepackage{amsthm,amsmath,natbib,amssymb}

\arxiv{arXiv:1201.2767}

\startlocaldefs

\def\qed{\hfill\quad \vrule height7.5pt width4.17pt depth0pt}

\newtheorem{theorem}{Theorem}[section]
\newtheorem{lemma}[theorem]{Lemma}
\newtheorem{cor}[theorem]{Corollary}

\newtheorem{prop}[theorem]{Proposition}

\newtheorem{remark}[theorem]{Remark}

\newcommand{\R}{{\mathbb{R}}}

\newcommand{\N}{{\mathbb{N}}}

\newcommand{\Z}{{\mathbb{Z}}}
\newcommand{\F}{{\mathcal{F}}}

\def\1{\mathbf{1}}

\newcommand{\vep}{\varepsilon}

\endlocaldefs

\begin{document}

\begin{frontmatter}

\title{On the Boundary of the Support of Super-Brownian Motion: With Appendices}
\runtitle{Boundary of the Support of SBM}

\begin{aug}

\author{\fnms{Carl} \snm{Mueller}\corref{}\thanksref{t1,m1}\ead[label=e1]{carl.2013@outlook.com}}
\author{\fnms{Leonid} \snm{Mytnik}\thanksref{t2,m2}\ead[label=e2]{leonid@ie.technion.ac.il}}
\and
\author{\fnms{Edwin} \snm{Perkins}\thanksref{t3,m3}\ead[label=e3]{perkins@math.ubc.ca}}

\thankstext{t1}{Supported by an NSF grant.}
\thankstext{t2}{Supported by an ISF grant.}
\thankstext{t3}{Supported by an NSERC Discovery Grant.}
\runauthor{Mueller, Mytnik, and Perkins}

\affiliation{University of Rochester\thanksmark{m1} and The Technion\thanksmark{m2} and University of British Columbia\thanksmark{m3}}

\address{Department of Mathematics\\University of Rochester
\\Rochester, NY  14627
\\ \printead{e1}}

\address{Faculty of Industrial Engineering and Management,
\\Technion -- Israel Institute of Technology
\\Haifa 32000, Israel
\\ \printead{e2}}

\address{Department of Mathematics
\\University of British Columbia
\\Vancouver, B.C., Canada V6T 1Z2
\\ \printead{e3}}

\end{aug}

\begin{abstract}
We study the density $X(t,x)$  of one-dimensional super-Brownian motion and find the asymptotic behaviour of $P(0<X(t,x)\le a)$ as $a\downarrow 0$ as well as the Hausdorff dimension of the boundary of the support of $X(t,\cdot)$.  The answers are in terms of the lead eigenvalue of the Ornstein-Uhlenbeck generator with a particular killing term.  This work is motivated in part by questions of pathwise uniqueness for associated stochastic partial differential equations.
\end{abstract}

\begin{keyword}[class=MSC]
\kwd[Primary ]{60H15, 60J68}
\kwd[; secondary ]{35K15}
\end{keyword}

\begin{keyword}
\kwd{super-Brownian motion}
\kwd{Hausdorff dimension}
\kwd{stochastic partial differential equations}
\end{keyword}

\end{frontmatter}

\section{Introduction}
\label{intro}
\setcounter{equation}{0}
We consider the jointly continuous density $X(t,x)$ ($t>0,x\in\R$)  of one-dimensional super-Brownian motion given by the unique in law solution of 
\begin{equation}\label{SPDE} \frac{\partial X(t,x)}{\partial t}=\frac{1}{2}\frac{\partial^2 X(t,x)}{\partial x^2}+\sqrt{X(t,x)}\dot W(t,x),\quad X\ge0.\end{equation}
Here $\dot W$ is a space-time white noise, $X_0$ is in the space $M_F(\R)$ of finite measures on the line, and $(X_t,t>0)$ is a continuous process taking values in the space $C_K(\R)$ of continuous functions with compact support in $\R$ (see, for example, Section III.4 of \cite{per02} for these results and the meaning of \eqref{SPDE}).
 We abuse notation slightly and also let $X_t(A)=\int_AX(t,x) dx$ denote the continuous $M_F(\R)$-valued process with density $X(t,\cdot)$ for $t>0$, i.e., the associated super-Brownian motion. 
 
 Our goal is to study the boundary of the zero set of $X_t$, or equivalently the boundary of the support of $X_t$, given by
 \begin{equation}\label{BZdef2a}
BZ_t=\partial(\{x:X(t,x)=0\})=\{x:X(t,x)=0, \forall \delta>0\, X_t((x-\delta,x+\delta))>0\}.
\end{equation}
The two related questions we consider are:\hfil\break
1. How large is $BZ_t$? For example, what is its Hausdorff dimension?\hfil\break
2. What is the asymptotic behaviour of $P(0<X(t,x)<a)$ as $a\downarrow 0$?

As super-Brownian motion models a population undergoing random motion and critical reproduction, a detailed understanding of the interface between the population and 
empty space gives a snapshot of how the population ebbs and flows.  Moreover the answers we found are not what was originally expected.  Standard estimates  show that $X(t,\cdot)$ is locally H\"older continuous of index $1/2-\vep$ for any $\vep>0$ (see Proposition~\ref{classmod} below). But near the zero set of $X(t,\cdot)$, one can expect more regular behavior as the noise in \eqref{SPDE} is mollified.  In fact \cite{mps06} essentially showed that near the zero set of $X(t,\cdot)$ the density is locally H\"older continuous of any index less than one (see Proposition~\ref{improvedmod} below for a precise statement).  The increased regularity led to independent conjectures 14 years ago (one by Carl Mueller and Roger Tribe and the other by Ed Perkins and Yongjin Wang) of the following:

\medskip
\noindent{\bf Conjecture A.} The Hausdorff dimension of $BZ_t$ is zero a.s. 
\medskip

\noindent This also was spurred on by wishful thinking as such a result would
help prove pathwise uniqueness in equations such as \eqref{SPDE}, as we explain next.

The connection with pathwise uniqueness in stochastic pde's with non-Lipschitz coefficients is one reason for our interest in these questions.  Mass moves with a uniform modulus of continuity in equations such as $\eqref{SPDE}$ and more generally in 
\begin{equation}\label{SPDEa} \frac{\partial X(t,x)}{\partial t}=\frac{1}{2}\frac{\partial^2 X(t,x)}{\partial x^2}+X(t,x)^\gamma\dot W(t,x),\ X\ge0,\end{equation}
 for $0<\gamma<1$ (see Theorem~3.5 of \cite{mp92}).
This means one can localize the evolution of solutions to \eqref{SPDEa} in space and so if pathwise uniqueness fails then one expects that the solutions $X,Y$ which separate at time $T$ say, will initially separate at points 
in $BZ_T(X)\cap BZ_T(Y)$, where we have introduced dependence on the particular process.  This is because in the interior of the support, say where $X_T\ge \eta>0$, we have 
Lipschitz continuous coefficients and so solutions should coincide for a positive time due to the uniform modulus of continuity, and in the interior of the zero set solutions will remain at zero for some positive length of time thanks to the same reasoning (and lack of any immigration terms).  
As a result one expects that the larger $BZ_t$ is, the easier it is for solutions to separate, and so the less likely pathwise uniqueness is.  This reasoning is of course heuristic but here are some 
precise illustrations of the principle.

\begin{theorem}\label{pnud}(\cite{yc15}) Let $b:\R\to\R_+$ be a smooth function with support $[0,1]$.  Then pathwise uniqueness fails in 
\begin{equation}\label{SPDEI} \frac{\partial X(t,x)}{\partial t}=\frac{1}{2}\frac{\partial^2 X(t,x)}{\partial x^2}+\sqrt{X(t,x)}\dot W(t,x)+b.\end{equation}
\end{theorem}
\noindent One step in the proof is to show that $BZ_t\cap [0,1]$ has positive Lebesgue measure with positive probability. This proceeds by first fixing $x\in(0,1)$ and using a Poisson point process calculation to show that $P(X(t,x)=0)>0$. It is then easy to see that the 
$b$-immigration forces $X_t((x-\delta,x+\delta))>0$ a.s. for each $\delta>0$. The proof of the theorem then goes on to show that solutions $X,Y$ can separate in $\cup_{t\in[0,1]}BZ_t(X)\cap BZ_t(Y)$.  In short the presence of the immigration $b$ changes the nature of the boundary of the support and makes it possible to establish pathwise non-uniqueness.  

On the uniqueness side of things, two of us conjectured that the methods of \cite{mp11} would allow one to establish 
\begin{align}\label{uniqconj}&\text{If for some $\vep>0$, $P(0<X(t,x)\le a)\le Ca^{1+\vep}$,}\cr
&\text{then pathwise uniqueness would hold in \eqref{SPDE}.}\end{align}

\noindent We never tried to write out a careful proof of the implication in part because we believed the correct answer to Question 2 above was 

\medskip
\noindent {\bf Conjecture B.} $P(0<X(t,x)\le a)=O(a)$ (which is consistent with Conjecture A above). 
\medskip

\noindent Nonetheless, this would be a rather nice state of affairs as it would suggest 
that the $\gamma=1/2$ case of \eqref{SPDEa} is critical and one could expect pathwise 
uniqueness to hold for $\gamma>1/2$. It is natural to expect that the $\gamma=1/2$ case would then require additional work just as in the classical SDE counterpart resolved by
Yamada and Watanabe 45 years ago (and unlike the signed case for general SPDE's where $3/4$-H\"older continuity in the solution variable is critical by \cite{mp11} and \cite{mmp14}).

Our main results on Questions 1 and 2 will show both Conjectures A and B are in fact false. 
To describe them, for $\lambda>0$, let $V(t,x)=V^\lambda(t,x)$ be the unique solution of 
\begin{equation}\label{Vlambdapde}\frac{\partial V}{\partial t}=\frac{1}{2}\frac{\partial^2 V}{\partial x^2}-\frac{1}{2}V^2,  \quad V(0,x)=\lambda\delta_0(x),
\end{equation}
where $V$ is $C^{1,2}$ on $[0,\infty)\times\R\setminus\{(0,0)\}$. (See \cite{BPT85}, \cite{MP03} and the references therein.)
A simple scaling argument shows that 
\begin{equation}\label{pdescale}
V^{\lambda r}(s,x)=\lambda^2V^r(\lambda^2s,\lambda x)\quad\forall r,\lambda,s>0,\, x\in\R.
\end{equation}
If $E_\mu$ denotes expectation for $X$ starting at $X_0=\mu$, then we have (see, for example, Theorem 1.1 in \cite{MP03} and the references there)
\begin{equation}\label{LLE}E_{\delta_0}(e^{-\lambda X(t,x)})=E_{\delta_x}(e^{-\lambda X(t,0)})=e^{-V^\lambda(t,x)},
\end{equation}
and by the multiplicative property,
\begin{equation}\label{LLEb}E_{X_0}(e^{-\lambda X(t,0)})=\exp\Bigl(-\int V^\lambda(t,y)\,dX_0(y)\Bigr).
\end{equation}
So from the above we see that $V^\lambda(t,x)\uparrow V^\infty(t,x)$ as $\lambda\uparrow \infty$, where 
\begin{equation}\label{Vinfform}P_{\delta_0}(X(t,x)=0)=P_{\delta_x}(X(t,0)=0)=e^{-V^\infty(t,x)}
\end{equation}
 and so 
 \begin{equation}\label{Vinftybound}V^\infty(t,x)\le2/t<\infty\end{equation} 
 because $P_{\delta_0}(X_t\equiv0)=\exp(-2/t)$ (see, e.g., (II.5.12) in \cite{per02}).
If $r\to\infty$ in \eqref{pdescale} with $\lambda^2=s^{-1}$, we get 
\begin{equation}\label{VF} V^\infty(s,x)=s^{-1}F(xs^{-1/2}),
\end{equation}
where $F(x)=V^\infty(1,x)$.  The function $F$ has been studied in the pde literature and can be intrinsically characterized as the solution of an ode. This, and other properties of $F$, are recalled in Section~\ref{FSec}. For now we only need to know that it is a symmetric $C^2$ function on  the line which vanishes at infinity.  Let $Lh(x)=\frac{h''}{2}-\frac{x}{2}h'(x)$ be the generator of the Ornstein-Uhlenbeck process, let $m$ be the standard normal distribution on the line,  and set $L^F(h)=Lh-Fh$.  By standard Sturm-Liouville theory (see Theorem~\ref{fens}) there is a complete orthonormal system for $L^2(m)$ consisting of $C^2$ eigenfunctions for $L^F$, $\{\psi_n:n\in\Z_+\}$, with corresponding negative eigenvalues $\{-\lambda_n\}$ where $\{\lambda_n\}$ is nondecreasing. The largest eigenvalue $\lambda_0$ is simple and satisfies $1/2<\lambda_0<1$. The latter and a bit more is proved in Proposition~\ref{lam0}.  

Here are our answers to the above questions. $\text{dim}(A)$ denotes the Hausdorff dimension of a set $A\subset\R$.
In the next two results $X(t,x)$ is the density of super-Brownian motion satisfying \eqref{SPDE} starting with finite initial measure $X_0$ and $BZ_t$ is defined as in\eqref{BZdef2a}. 
\begin{theorem}\label{lefttailSBM}
(a) For some $C_{\ref{lefttailSBM}}$, for all $a,t>0$, and $x\in\R$,
\[P_{X_0}(0<X(t,x)\le a)\le C_{\ref{lefttailSBM}}X_0(\R)t^{-(1/2)-\lambda_0}a^{2\lambda_0-1}.\]
(b) For all $K\in\N$ there is a $C(K)>0$ so that if $X_0(\R)\le K$, $X_0([-K,K])\ge K^{-1}$, $t\ge K^{-1}$ and $|x|\le K$, then
\[P_{X_0}(0<X(t,x)\le a)\ge C(K)t^{-(1/2)-\lambda_0}a^{2\lambda_0-1}\quad\text{for all }0<a\le \sqrt{t}.\]
\end{theorem}


\begin{theorem}\label{dimBZ} For all $X_0\neq 0$ and $t>0$,\hfil\break
(a) $\text{dim}(BZ_t)\le 2-2\lambda_0$ $P_{X_0}$-a.s.\\
(b)  $\text{dim}(BZ_t)= 2-2\lambda_0$ with positive $P_{X_0}$-probability.
\end{theorem}

Note in the above results that both $2\lambda_0-1$ and $2-2\lambda_0$ are in $(0,1)$ and so these results do disprove Conjectures A and B.
In Theorem~\ref{dimBZ}(b) one expects that $\text{dim}(BZ_t)=2-2\lambda_0$ a.s. on $\{X_t\neq 0\}$ but we have not been able to show this. 

 \begin{remark} Both the above results extend immediately to solutions of 
 \begin{equation}\label{genSPDE}\frac{\partial X(t,x)}{\partial t}=\frac{\sigma^2}{2}\frac{\partial^2 X(t,x)}{\partial x^2}+\sqrt{\gamma X(t,x)}\dot W(t,x).
 \end{equation}
 where the constants in Theorem~\ref{lefttailSBM} now may depend on $\gamma,\sigma^2>0$.
 This is clear since a simple scaling result shows that if $X$ is the solution of \eqref{SPDE}, then $\gamma\sigma^{-2}X(\sigma^2t,x)$ has the unique law of any solution of 
\eqref{genSPDE}. 
\end{remark}

Theorem~\ref{lefttailSBM} is contained in Theorem~\ref{SBMdensity} below. A Tauberian
theorem will show that $P_{\delta_0}(0<X(t,x)\le a)\sim a^\alpha$ as $a\downarrow 0$ if and only if $$E_{\delta_0}\Bigl(e^{-\lambda X(t,x)}1(X(t,x>0)\Bigr)\sim \lambda^{-\alpha}\text{ as }\lambda\uparrow\infty.$$ Here $\sim$ means bounded above and below by positive constants and $\alpha=2\lambda_0-1$.  If we use \eqref{LLEb} and \eqref{Vinfform}, this becomes
\[e^{-V^\lambda(t,x)}-e^{-V^\infty(t,x)}\sim\lambda^{-\alpha}\quad\text{ as }\lambda\uparrow\infty,\]
and using \eqref{Vinftybound} this reduces to
\begin{equation}\label{}V^\infty(t,x)-V^\lambda(t,x)\sim\lambda^{-\alpha}\quad\text{as }\lambda\uparrow\infty.
\end{equation}
We have not been careful with dependence on $t$ or $x$ but by Dini's theorem we know that $\lim_{\lambda\to\infty}V^\lambda(t,x)=V^\infty(t,x)$ uniformly for $(t,x)$ in compact
subsets of $[0,\infty)\times\R\setminus\{(0,0)\}$--this is Theorem 1 of \cite{KP85}.  Evidently to prove Theorem~\ref{lefttailSBM} we need a rate of convergence  in the Kamin and Peletier result, and Proposition~\ref{Vrate} will give the following which may be of interest to pde specialists and will be used to complete the proof of Theorem~\ref{lefttailSBM}.

\begin{theorem}\label{Vcvgcerate}
There are positive constants $\bar C$ and for each $K\ge 1$, $\underline C(K)$, such
that  for all $t>0$,
\begin{align}\label{Vcrate2}
\sup_x V^\infty(t,x)-V^\lambda(t,x)\le \bar Ct^{-\frac{1}{2}-\lambda_0}\lambda^{-(2\lambda_0-1)}\ \forall\lambda>0,
\end{align}
and
\begin{align}\label{Vcrate}
\underline C(K)t^{-\frac{1}{2}-\lambda_0}\lambda^{-(2\lambda_0-1)}\le \inf_{|x|\le K\sqrt t} V^\infty(t,x)-V^\lambda(t,x)\ \forall \lambda\ge t^{-1/2}.\end{align}
\end{theorem}

The lower bound in Theorem~\ref{dimBZ}(b) is established in Section~\ref{lbondim} (see Theorem~\ref{dimlb}) by first establishing a capacity condition for a set to be non-polar for $BZ_t$ (Corollary~\ref{polbz}) through a Frostman-type argument, and then taking the above set to be the range of an appropriate subordinator and using the potential theory for the subordinator (a well-known trick). The key to the above capacity condition is a second moment bound (Proposition~\ref{secmombnd}) which is established in Section~\ref{5.1}.  The corresponding upper bound is proved in Theorem~\ref{dimubnd} in Section~\ref{dimubound} by first modifying the proof of Theorem~\ref{lefttailSBM} to get a bound on $P_{X_0}(0<X_t([x,x+\vep])\le \vep^2M)$ (Theorem~\ref{lowdensint} which is proved in Section~\ref{5.8}). One then uses this and the improved modulus of continuity for $X$ near its zero set (Proposition~\ref{improvedmod}) to carry out a standard covering argument which in fact bounds the box dimension.

Although the failure of Conjectures A and B indicate that our approach may
not shed light on the pathwise of uniqueness of solutions to \eqref{SPDE} we 
remain optimistic that progress can be made on \eqref{SPDEa} for some values
of $\gamma$ in $(1/2,3/4]$. In fact \eqref{uniqconj} was a special case of the conjecture
\begin{align}\label{uniqconjg}&\text{if for some $\alpha>3-4\gamma$, $P(0<X(t,x)\le a)\le Ca^{\alpha}$,}\cr
&\text{then pathwise uniqueness would hold in \eqref{SPDEa}.}\end{align}
Note for $\gamma>3/4$ one can take $\alpha=0$ but in this case pathwise uniqueness
is a special case of the main result in \cite{mp11} whose ideas underly our heuristic
proof of \eqref{uniqconjg}.  There is also an exponential dual to solutions
of \eqref{SPDEa} for $\gamma\in[1/2,1)$ (see \cite{myt98w}) and we believe the methods of this paper can be used to resolve the lefthand tail asymptotics for solutions of \eqref{SPDEa} as well.  We expect power tail behaviour for all $\gamma$, and so by \eqref{uniqconjg}, for $\gamma<3/4$ but close to $3/4$,  
we conjecture that pathwise uniqueness will hold in \eqref{SPDEa}.  
Hence, although the $3/4$-H\"older condition in \cite{mp11} is sharp for pathwise uniqueness in general (by \cite{mmp14}), it would not be sharp for the family of non-negative solutions to \eqref{SPDEa}.  \\

\noindent{\bf Convention.} We will use $E^Z_{a}$ to denote expectation for the Markov process $Z$ starting at a point $a$ and (abusing the notation slightly) use $E^Z_\mu$ to denote the corresponding expectation where $Z_0$ now has law $\mu$--sometimes $\mu$ will be only a finite measure.  
 
 \section{Eigenfunction Expansions}\label{sec:efn}
\setcounter{equation}{0}
Let $Y_t$ denote the Ornstein-Uhlenbeck process associated with the infinitesimal generator $L$. In this section we usually will drop the $Y$ in the notation $E^Y_x$.  Its semigroup is denoted by $P_t$ and its resolvent by $R_\lambda$.  If $\phi\in C[-\infty,\infty]$ (the space of continuous functions with finite limits at $\pm\infty$) with $\phi\ge 0$, we let $L^\phi h=Lh-\phi h$, the generator associated with the diffusion, $Y^\phi$, obtained by killing $Y$ at time $\rho_\phi=\inf\{t:\int_0^t\phi(Y_s)\,ds>e\}$, where $e$ is an independent exponential variable.  We denote its semigroup and resolvent by $P^\phi_t$ and $R_\lambda^\phi$ ($\lambda>0$), respectively.  The above semigroups are strongly continuous contraction semigroups on $L^2(m)$.  The contraction part is elementary as $m$ is stationary for $Y$. For the strong continuity see Lemma~\ref{sgc} below for $P_t$ and it is easy to check that $\lim_{t\to 0}\Vert P_tf-P^\phi_t f\Vert_2=0$ for all $f\in L^2(m)$.  For now we will consider $L$ and $L^\phi$ defined on $D=\{h\in C^2\cap L^2(m): Lh\in L^2(m)\}$.  

\begin{lemma}\label{sgc} For all $f\in L^2(m)$, $\lim_{t\downarrow 0}E_m((f(Y_t)-f(Y_0))^2)=0$. 
\end{lemma}
\begin{proof}  As $Y_t$ is stationary under $P_m$, a standard approximation argument allows us to assume $f$ is continuous with compact support.  The result now follows by Dominated Convergence. 
\end{proof}
We let $L_0^\phi$ and $L_0$ denote the infinitesimal generators of the $L^2(m)\equiv L^2$-semigroups $P_t^\phi$ and $P_t$, respectively, on their domains $D(L_0^\phi)$ and $D(L_0)$, respectively.  So, for example 
$$D(L_0)=\{f\in L^2:\exists L_0f\in L^2\text{ such that }\lim_{t\downarrow 0}\Vert(P_tf-f)/t-L_0f\Vert_2=0\}.$$ The subscript $0$ is a temporary measure to avoid confusion which we address now. 

\begin{lemma} \label{gens}(a) $D(L_0^\phi)=D(L_0)$ and $L_0^\phi f=L_0f-\phi f$ for all $f\in D(L_0)$.\\
(b) $L_0^\phi$is an extension of the differential operator $L^\phi$, the latter on $D$.
\end{lemma}
\begin{proof} (a) This is a routine calculation. The fact that $\phi\in C[-\infty,\infty]$ helps here. 

\noindent(b) By the above we may assume $\phi=0$.  If $f\in D$, then $M_f(t)=f(Y_t)-f(Y_0)-\int_0^tLf(Y_s)\,ds$ is a square integrable martingale under $P_m$.  Therefore
\begin{eqnarray*}\Vert (P_tf-f)/t-Lf\Vert_2^2&=&\int\Bigl(E_x\Bigl(\int_0^tLf(Y_s)\,ds\Bigr)/t-Lf(x)\Bigr)^2dm(x)\\
&\le&\int_0^tE_m((Lf(Y_s)-Lf(Y_0))^2)ds/t.
\end{eqnarray*}
The last expression approaches $0$ as $t\to0+$ by the previous Lemma, and the result follows.
\end{proof}
 Henceforth we will drop the subscript $0$'s on $L_0^\phi$ (in view of the above result this should cause no confusion). 

Here is the result we will need to describe our main results. 
In (e), $C([0,\infty),\R)$ is the usual space of continuous paths with the topology of uniform convergence on bounded time sets.

\begin{theorem}\label{fens} (a) There is a complete orthonormal system (cons), $\{\psi_n:n\in \Z_+\}$, of $C^2$ eigenfunctions in $L^2(m)$ for $L^\phi$ satisfying $L^\phi\psi_n=-\lambda_n\psi_n$, where $\{\lambda_n\}$ is a nondecreasing non-negative sequence diverging to $\infty$, ${-\lambda_0}$ is a simple eigenvalue and $\psi_0>0$.\\
(b) $R_\lambda^\phi$ is a symmetric Hilbert-Schmidt integral operator on $L^2(m)$. There is a jointly continuous symmetric kernel $G_\lambda^\phi:\R^2\to [0,\infty)$ such that 
$$R_\lambda^\phi h(x)=\int G_\lambda^\phi(x,y)h(y)\,dm(y),$$
and 
$$G_\lambda^\phi(x,y)=\sum_{n=0}^\infty \frac{1}{\lambda+\lambda_n}\psi_n(x)\psi_n(y),$$
where the series converges in $L^2(m\times m)$ and uniformly absolutely on compacts. \\
(c) The killed diffusion $Y^\phi$ has a jointly (in $(t,x,y)$) continuous transition density, $q^\phi(t,x,y)\equiv q(t,x,y)$, for $t>0$, given by 
$$q(t,x,y)=\sum_{n=0}^\infty e^{-\lambda_n t}\psi_n(x)\psi_n(y),$$
where the convergence is in $\L^2(m\times m)$ and uniformly absolutely 
for $(t,x,y)\in [\vep,\infty)\times[\vep,\vep^{-1}]^2$ for any $\vep>0$. Moreover, if $0<\delta<1/4$, and $s^*=s^*(\delta)>0$ satisfies
\begin{equation}\label{s*def}2\delta=\frac{e^{-s^*/2}-e^{-s^*}}{1-e^{-s^*}}
\end{equation}
($s^*$ will increase to $\infty$ as $\delta\downarrow 0$), then there is a $c(\delta)$ such that
\begin{equation}\label{qbound1}
q(t,x,y)\le c(\delta)e^{-\lambda_0 t}\exp(\delta(x^2+y^2))\text{ for all }t\ge s^*(\delta).
 \end{equation}
(d) If $\theta=\int\psi_0\,dm$, then for any $\delta>0$ there is a $c_\delta>0$ such that for all $t\ge0$ and $x\in\R$,
\begin{equation}\label{rhotail}
e^{\lambda_0 t}P_x(\rho_\phi>t)=\theta\psi_0(x)+r(t,x),
\end{equation}
where
\begin{equation}\label {rbound}
|r(t,x)|\le c_\delta e^{\delta x^2}e^{-(\lambda_1-\lambda_0)t},
\end{equation}
\begin{equation}\label{psi0bnd}
\psi_0(x)\le c_\delta e^{\delta x^2},
\end{equation}
and for $t\ge s^*(\delta)$, 
\begin{equation}\label{rboundnew}
|r(t,x)|\le \sum_1^\infty e^{-(\lambda_n-\lambda_0)t}|\psi_n(x)|\int|\psi_n|dm\le c_\delta e^{\delta x^2}e^{-(\lambda_1-\lambda_0)t}.
\end{equation}

(e) As $T\to\infty$, $P_x(Y\in\cdot|\rho_\phi>T)\rightarrow P_x^\infty$ weakly on $C([0,\infty),\R)$ where $P_x^\infty$ is the law of the diffusion with transition density (with respect to $m$),
\begin{equation}\label{conddiff}\tilde q(t,x,y)\equiv q(t,x,y)\frac{\psi_0(y)}{\psi_0(x)}e^{\lambda_0 t}.
\end{equation}
\end{theorem}
\begin{proof} See Appendix~\ref{efunctionexpa}.
\end{proof}

\section {A Nonlinear Differential Equation and Some Associated Eigenvalues}\label{FSec}
\setcounter{equation}{0}
Recall that $F(x)=V^\infty(1,x)$. We start by recording the convergence results from Theorems 1 and 2 of \cite{KP85} which were discussed in Section~\ref{intro}. Part (b) in fact is immediate from (a) and \eqref{pdescale}.
\begin{prop}\label{kamin}
(a) $\lim_{\lambda\to\infty}V^\lambda(t,x)=V^\infty(t,x)$, where the convergence is uniform on compact subsets of $S$.\\
(b) For any $\lambda>0$ and $a>0$, 
$$\lim_{t\to \infty} \sup_{|x|\le a}|tV^\lambda(t,xt^{1/2})-F(x)|=0.$$
\end{prop}
In the pde literature $F:\R\to[0,\infty)$ is characterized as the unique solution of the following differential equation:
\begin{eqnarray}
\nonumber&(i)& \frac{F''(y)}{2}+\frac{y}{2}F'(y)+F(y)-\frac{F^2(y)}{2}=0,\\
\label{ode}&(ii)& F>0, F\text{ is } C^2\text{ on }\R,\\
\nonumber&(iii)&F'(0)=0, F\sim c_0 y e^{-y^2/2}\text{ as }y\to\infty.
\end{eqnarray}
Here $\sim$ means the ratio goes to one as $y\to\infty$ and $c_0>0$.  
This result follows from \cite{BPT85}, with $f:[0,\infty)\to[0,\infty)$ satisfying equations (1.7)-(1.9) of that reference (with $N=1$ and $p=2$ in our setting), where $F(y)=2f(\sqrt 2 y)$ for $y\ge 0$ and we extend $F$ to the line by symmetry.   The above ode is then immediate from the Theorem following (1.9) in \cite{BPT85} and the trivial fact that the condition $F'(0)=0$ and fact that $F$ is $C^2$ on the positive half-line ensures the symmetric extension is $C^2$ on the line. In fact uniqueness holds if the strong asymptotic  in condition (iii) is replaced with
\begin{eqnarray*}&(iii)'& F'(0)=0,\quad \lim_{y\to\infty}y^2F(y)=0.
\end{eqnarray*}
Here are some additional properties of $F$. 
\begin{lemma}\label{Fprop} (a) For all $y\ge x_0\ge 0$,
$$F'(y)=\exp\{-y^2/2+x_0^2/2\}F'(x_0)+\int_{x_0}^y \exp\{-y^2/2+z^2/2\}F(z)(F(z)-2)\,dz.$$
(b) $\lim_{y\to\infty}e^{y^2/2}\frac{F'(y)}{y^2}=-c_0$, where $c_0$ is as in \eqref{ode}(iii).\\
(c) $1<F(0)<2$.\\
(d) $F$ is strictly decreasing on $[0,\infty)$. 
\end{lemma}
\begin{proof} The differential equation \eqref{ode}(i) may be rewritten as $$(e^{z^2/2}F'(z))'=e^{z^2/2}F(z)(F(z)-2).$$ (a) follows easily. To derive (b), take $x_0$ large and then use the asymptotics from \eqref{ode}(iii).
For (c), note by (a) with $x_0=0$, $F$ is increasing until $F\le 2$. So if $F(0)>2$, it can never pass below $2$, a contradiction.  If $F(0)=2$, then by uniqueness to the initial value problem, $F\equiv2$, another contradiction.  
It now follows from (a) with $x_0=0$ that $F'<0$ for positive values until $F$ hits $2$ but evidently this can therefore never happen. This proves (d).  It remains to prove $F(0)>1$.  A simple calculation using \eqref{ode}(i) gives $(yF+F')'=F(F-1)$. Integrating over the line we get $\int_{\R} F^2\,dy=\int_{\R} F\,dy$.  If $F(0)\le1$, then by (d), $0<F<1$ on $(0,\infty)$ which contradicts the above equality of integrals. (Note that the Remark prior to Lemma 11 in \cite{BPT85} gives $F(0)\ge 1$.)
\end{proof}

In the pde literature $V^\infty(t,x)$ given by \eqref{VF} is called a very singular solution 
of the heat equation with absorption.  One can easily check (or see Section 1 of \cite{BPT85}) that $V=V^\infty$ 
is a $C^{1,2}$ (on $S=[0,\infty)\times\R-\{(0,0)\}$) solution of
\begin{eqnarray}  \label{pde1}&(i)& \frac{\partial V}{\partial t}=\frac{1}{2}\frac{\partial^2 V}{\partial x^2}-\frac{1}{2}V^2\\
\nonumber&(ii)& V(0,x)=0\text{ for all }x\neq 0; \ \lim_{t\to 0}\int_{\R} V(t,x)\,dx=\infty.
\end{eqnarray}

 Recall that  $X_t(dx)=X(t,x)dx$, where $X$ solves \eqref{SPDE}.
Translation invariance and \eqref{LLEb} imply that 
\begin{equation}\label{LLEe}E^X_{X_0}(e^{-\lambda X(t,x)})=\exp\Bigl(-\int V^\lambda(t,y-x)\,dX_0(y)\Bigr).
\end{equation}
We let $\N_x$ denote the canonical measure associated with $X$ starting at $\delta_x$ (see Section II.7 of \cite{per02}).  It is an easy consequence of Theorem~II.7.2 of the latter reference that 
\begin{equation}\label{can} 
\exp\Bigl(-\int 1-e^{-\lambda X(t,0)}d\N_x\Bigr)=e^{-V^\lambda(t,x)}.
\end{equation}
\begin{prop}\label{superF} For all $x\in \R$, 
$e^{-F(x)}=P_{\delta_x}(X(1,0)=0)$ and \hfil\break
$F(x)=\N_x(X(1,0)>0)$.
\end{prop}
\begin{proof} The first equality is immediate from \eqref{Vinfform}. 
Let $t=1$ and $\lambda\uparrow\infty$ in \eqref{can} to derive the second. \end{proof}
We recall the general exponential duality which underlies the above  (see, e.g., Theorem II.5.11 of \cite{per02}).  
For $\phi$ non-negative, bounded and measurable, let $V(t,x)=V(\phi)(t,x)$ be the unique mild solution of 
\begin{equation}\label{semipde}
\frac{\partial V}{\partial t}=\frac{1}{2}\frac{\partial^2 V}{\partial x^2}-\frac{1}{2}V^2,\quad V_0=\phi.
\end{equation}
If $X_t(\phi)=\int\phi dX_t$, then,
\begin{equation}\label{LLEc}
E^X_{X_0}(\exp(-X_t(\phi))=\exp(-X_0(V_t(\phi)).
\end{equation}
If $\phi\in C^2_b(\R)$ (functions with continuous bounded partials of order less than $2$) then $V(t,x)$ has continuous bounded derivatives of order up to $1$ in $t$ and $2$ in $x$, and \eqref{semipde} holds in the classical (i.e., pointwise) sense.  

Now return to the eigenfunction expansions of Section~\ref{sec:efn} in
the case where $\phi=F$ or $F/2$.  We denote dependence on $\phi$ by $\lambda_0(\phi)$ and $\psi_0^\phi$, and $\langle\cdot,\cdot\rangle$ is the inner product in $L^2(m)$. 

\begin{prop}\label{lam0} (a) $\lambda_0(F/2)=\frac{1}{2}$ and the corresponding eigenfunction is $\psi^{F/2}_0(x)=c_F e^{x^2/2}F(x)$, where $c_F>0$ is a normalizing constant.\\
(b) $\frac{1}{2}<\lambda_0(F)<\1$.  More precisely,
$$\frac{1}{2}+\frac{1}{2}\int F(x)(\psi_0^F)^2(x)dm(x)\le \lambda_0(F)\le 1-\frac{1}{2}\int (c_Fe^{x^2/2}F(x))'(x)^2\,dm(x).$$
\end{prop}
\begin{proof} Let $\psi(x)=e^{x^2/2}F(x)\in C^2\cap L^2(m)$ (the latter
by \eqref{ode}(iii)). Then 
\begin{eqnarray*}L^{F/2}\psi&=&\frac{1}{2}(\psi'e^{-x^2/2})'e^{x^2/2}-\frac{F}{2}\psi\\
&=&e^{x^2/2}\Bigl[\frac{F''}{2}+\frac{xF'}{2}+\frac{F}{2}-\frac{F^2}{2}\Bigr]\\
&=&-e^{x^2/2}\Bigl(\frac{F}{2}\Bigr),
\end{eqnarray*}
the last by \eqref{ode}(i). This shows $\psi\in D\subset D(L)$ and $L^{F/2}\psi=-\frac{1}{2}\psi$.  Recall by Theorem~\ref{fens}(a), the eigenfunction corresponding to the simple eigenvalue $-\lambda_0(F/2)$ is positive, as is $F$. By orthogonality of eigenfunctions corresponding to distinct eigenvalues we must therefore have $\lambda_0(F/2)=\frac{1}{2}$ and $\psi_0^{F/2}=c\psi$, for some normalizing constant $c>0$.  

\noindent(b) The variational characterization of $\lambda_0$ gives
\begin{equation}\label{varprinc}
\lambda_0(F)=\min\{\langle-L^F\psi,\psi\rangle: \psi\in D(L),\ \Vert\psi\Vert_2=1\},
\end{equation}
where the minimum is attained at $\psi=\psi_0^F$.  (The latter is clear and to
see the former one can set $\psi=R^F_\lambda\phi$ and expand $\phi$ in terms of the basis $\psi_n$.)  If we set $\psi=\psi_0^{F/2}$ we therefore get
\begin{eqnarray*}
\lambda_0(F)\le\langle-L^F\psi,\psi\rangle
&=&2\langle -L^{F/2}\psi,\psi\rangle+\langle L\psi,\psi\rangle\\
&=&2\lambda_0(F/2)-\frac{1}{2}\int_{-\infty}^\infty \psi'(x)^2dm\\
&=&1-\frac{1}{2}\Vert\psi'\Vert_2^2<1,
\end{eqnarray*}
where the next  to last equality holds by an integration by parts (Lemma~\ref{Fprop}(b) handles the boundary terms), and the last equality holds by (a). 
Turning next to the lower bound on $\lambda_0(F)$, if $\psi_0=\psi_0^F$, we have, using the variational characterization of $\lambda_0(F/2)$,
\begin{eqnarray*}
\lambda_0(F)=\langle -L^F\psi_0,\psi_0\rangle&=&\langle-L^{F/2}\psi_0,\psi_0\rangle+\frac{1}{2}\int F\psi_0^2\,dm\\
&\ge& \lambda_0(F/2)+\frac{1}{2}\int F\psi_0^2\,dm
=\frac{1}{2}+\frac{1}{2}\int F\psi_0^2\,dm.
\end{eqnarray*}\end{proof}
{\bf Henceforth we will write $\psi_0$, $\lambda_n$ and $\rho$ for $\psi_0^F$, $\lambda_n(F)$ and $\rho_F$, respectively.}

\section{Asymptotics for Super-Brownian Motion at the Boundary of its Support}
\setcounter{equation}{0}
Recall that $X(t,x)$ is the density of super-Brownian motion which solves \eqref{SPDE}.  Define
\begin{equation}\label{Hdef} H_u(x)=uV^1(u,\sqrt ux)=V^{\sqrt u}(1,x)\uparrow V^\infty(1,x)=F(x),\ \text{as $u\uparrow \infty$},
\end{equation}
uniformly on compacts, where we have used \eqref{pdescale} in the second equation and Proposition~\ref{kamin} in the convergence. By \eqref{Vinftybound} one obtains the elementary inequality
\begin{equation}\label{expbnds}e^{-V^\lambda(t,x)}-e^{-V^\infty(t,x)}\le V^\infty(t,x)-V^\lambda(t,x)\le e^{2/t}(e^{-V^\lambda(t,x)}-e^{-V^\infty(t,x)}).
\end{equation}
\noindent{\bf Notation.} We let $p(t,x)=p_t(x)$ denote the standard Brownian density.
\medskip

In the following Lemma recall that $Y$ is the Ornstein-Uhlenbeck process starting at $x$ under $P_x$. 
\begin{lemma}\label{LTequation}
Let $h\ge 0$ be a bounded Borel measurable function on the real line,  let $B$ be a standard Brownian motion starting at $0$ under $P^B_0$, and set $T=\log(\lambda^2t)$.  Then for $\lambda^2t\ge 1$ and any finite initial measure $X_0$,
\begin{align}\label{LTdecomp}
&E^X_{X_0}\left(\int e^{-\lambda X(t,x)} h(x)X(t,x)dx\right)\cr
&=E^B_0\Bigl(\exp\Bigl(-\int_0^1 V^1(u,B(u))\,du\Bigr)E^Y_{B_1}\Bigl(\exp\Bigl(-\int_0^TH_{e^s}(Y_s)\,ds\Bigr)\cr
&\times\int \Bigl[h(w_0+\sqrt tY_T)\exp\Bigl(-\frac{1}{t}\int H_{e^T}(Y_T+\frac{w_0-x_0}{\sqrt t})dX_0(x_0)\Bigr)\Bigr]dX_0(w_0)\Bigr)\Bigr).
\end{align}
\end{lemma}
\begin{proof} Let $W_t$ be a Brownian motion starting with initial ``law" $X_0$ under the finite measure $E^W_{X_0}$. Apply the Campbell measure formula for $X_t$, or more specifically use Theorem~4.1.1 and then Theorem~4.1.3 of \cite{DP} with $\beta=1$ and $\gamma=1/2$ to see that
\begin{align}\label{campbell}
&E^X_{X_0}\Bigl(\int e^{-\lambda X(t,x)}h(x)X(t,x)dx\Bigr)\cr
&=E^W_{X_0}\times E^X_{X_0}\Bigl(h(W_t)\exp(-\lambda X(t,W_t))\exp\Bigl(-\int_0^t V^\lambda(t-s,W_s-W_t)\,ds\Bigr)\Bigr).
\end{align}
In the above we approximate $X(t,x)$ by $\int p_\vep(x-y)X_t(dy)$ and let $\vep\downarrow 0$ in order to apply Theorem~4.1.3 in \cite{DP}.  This limiting argument is easy to justify; use \eqref{LLEc} with $\phi^{\vep,\lambda}=\lambda p_\vep$ and the bound $V(\phi^{\lambda,\vep})(t-s,x)\le  \lambda p_{\vep+t-s}(x)\le \lambda(t-s)^{-1/2}$ to take the limit through the Lebesgue integral in $s$. 
Now use \eqref{LLEe} and then the scaling \eqref{pdescale} to conclude that
\begin{align}
&E^X_{X_0}\Bigl(\int e^{-\lambda X(t,x)}h(x)X(t,x)dx\Bigr)\cr
&=E^W_{X_0}\Bigl(h(W_t)\exp\Bigl(-\int_0^t V^\lambda(t-s,W_t-W_s)\,ds-\int V^\lambda(t,W_t-x_0)dX_0(x_0)\Bigr)\Bigr)\cr
\label{LT1}&=E^W_{X_0}\Bigl(h(W_t)\exp\Bigl(-\int_0^t V^1(\lambda^2(t-s),\lambda(W_t-W_s))\lambda^2\,ds\cr
&\phantom{=E_W\Bigl(h(W_t)\exp\Bigl(-}\quad-\lambda^2\int V^1(\lambda^2 t, \lambda(W_t-x_0))dX_0(x_0)\Bigr)\Bigr).
\end{align} 

 If $\hat W_s=W_t-W_{t-s}$ and $B_u=\lambda \hat W_{\lambda^{-2}u}$ for $u\le \lambda^2t$, then under $E^W_{X_0}$ and conditional on $W_0$, $B$ is a Brownian motion starting from $0$.  Noting that $W_t=W_0+\lambda^{-1}B_{\lambda^2t}$, we may rewrite \eqref{LT1} as
\begin{align*}
E^W_{X_0}&\Bigl(h(W_0+\frac{1}{\lambda}B_{\lambda^2t})\exp\Bigl(-\int_0^{t\lambda^2}V^1(u,B_u)\,du\cr
&\phantom{\Bigl(h(W_0+\frac{1}{\lambda}B_{\lambda^2t})\exp\Bigl(}-\lambda^2\int V^1(\lambda^2t,B_{\lambda^2t}+\lambda(W_0-x_0))dX_0(x_0)\Bigr)\Bigr)\cr
=&E^W_{X_0}\Bigl(\exp\Bigl(-\int_0^1 V^1(u,B_u)\,du\Bigr)E^B_{B_1}\Bigl[h(W_0+\lambda^{-1}B_{\lambda^2t-1})\cr
&\times\exp\Bigl(-\int_0^{\lambda^2t-1}V^1(1+u,B_u)\,du\\
&\phantom{\times\exp\Bigl(-\int_0^{\lambda^2t-1}V^1}-\lambda^2\int V^1(\lambda^2t,B_{\lambda^2t-1}+\lambda(W_0-x_0))dX_0(x_0)\Bigr)\Bigr]\Bigr).
\end{align*}
Set $$Y_s=B(e^s-1)e^{-s/2},$$ which is an Ornstein-Uhlenbeck process starting at $B_1$ under $E^B_{B_1}$.  Then \eqref{LT1} equals
\begin{align*}
E^W_{X_0}&\Bigl(\exp\Bigl(-\int_0^1 V^1(u,B_u)\,du\Bigr)E^Y_{B_1}\Bigl[\exp\Bigl(-\int_0^Te^sV^1(e^s,e^{s/2}Y_s)\,ds\Bigr)\cr
&\quad\times h(W_0+\sqrt tY_T)\exp\Bigl(-\int\frac{e^T}{t}V^1(e^T,e^{T/2}Y_T+\frac{e^{T/2}}{\sqrt{t}}(W_0-x_0))dX_0(x_0)\Bigr)\Bigr]\Bigr)\cr
&=E^W_{X_0}\Bigl(\exp\Bigl(-\int_0^1 V^1(u,B_u)\,du\Bigr)E^Y_{B_1}\Bigl[\exp\Bigl(-\int_0^T H_{e^s}(Y_s)ds\Bigr)h(W_0+\sqrt t Y_T)\cr
&\phantom{=E_W\Bigl(\exp\Bigl(-\int_0^1}\times\exp\Bigl(-t^{-1}\int H_{e^T}\Bigl(Y_T+\frac{W_0-x_0}{\sqrt t}\Bigr)\,dX_0(x_0)\Bigr)\Bigr]\Bigr).
\end{align*}
Recalling that $W_0$ is independent of $B$ under $E^W_{X_0}$, we see that the above equals the required expression.
\end{proof}
We will use the following Lemma which will allow us to apply Lemma~\ref{LTequation} to first get a preliminary bound on $V^\infty-V^\lambda$ for large $\lambda$, and then reapply these results to get exact asymptotics.  
\medskip

\begin{lemma} \label{prelgen}Assume for some $r>1$ and $\underline\lambda\ge0$,
\begin{equation*}
E^X_{\delta_0}\Bigl(\int e^{-\lambda X(t,x)}X(t,x)\,dx\Bigr)\le C(t)\lambda^{-r}\ \forall\lambda>\underline\lambda.
\end{equation*}
Then 
\begin{equation*}
\sup_x [V^\infty(t,x)-V^\lambda(t,x)]\le e^{6/t}t^{-1/2}C(t/2)(r-1)^{-1}\lambda^{1-r}\ \forall\lambda> \underline\lambda.
\end{equation*}
\end{lemma}
\begin{proof} 
Recalling \eqref{LLE}, we have 
\begin{equation*}
e^{-V^\lambda(t,x)}-e^{-V^\infty(t,x)}=E^X_{\delta_0}\Bigl(e^{-\lambda X(t,x)}1(X(t,x)>0)\Bigr).
\end{equation*}
The lefthand side is Lebesgue integrable in $x$ (e.g., it is bounded by $V^\infty(t,x)$ which is Lebesgue integrable by \eqref{VF} and the asymptotics for $F$) and so
\begin{equation*}
\int e^{-V^\lambda(t,x)}-e^{-V^\infty(t,x)}\,dx=E^X_{\delta_0}\Bigl(\int e^{-\lambda X(t,x)}1(X(t,x)>0)\,dx\Bigr).
\end{equation*}
It is easy to differentiate with respect to $\lambda>0$ through the integrals on the righthand side and so conclude for any $\lambda>\underline\lambda$,
\begin{align}\label{Vderiv}-\frac{d}{d\lambda}\Bigl(\int e^{-V^\lambda(t,x)}-e^{-V^\infty(t,x)}\,dx\Bigr)&=E^X_{\delta_0}\Bigl(\int e^{-\lambda X(t,x)}X(t,x)\,dx\Bigr)\cr
&\le C(t)\lambda^{-r}.
\end{align}
For $\lambda>\underline\lambda$ integrate the above over $[\lambda,\infty)$ and so deduce from \eqref{expbnds} that for $\lambda>\underline\lambda$,
\begin{align}\label{lebintbnd}
\int V^\infty(t,x)-V^\lambda(t,x)\,dx&\le e^{2/t}\int e^{-V^\lambda(t,x)}-e^{-V^\infty(t,x)}\,dx\cr
&\le \frac{e^{2/t}C(t)}{r-1}\lambda^{1-r}.
\end{align}
Next use the Markov property of $X$ to see that for $\lambda>\underline\lambda$, 
\begin{align}\label{Vincrement}
e^{-V^\lambda(t,x)}-e^{-V^\infty(t,x)}&=E^X_{\delta_x}\Bigl(e^{-\lambda X(t,0)}1(X(t,0)>0)\Bigr)\cr
&=E^X_{\delta_x}\Bigl(E^X_{X(t/2)}\Bigl(e^{-\lambda X(t/2,0)}1(X(t/2,0)>0)\Bigl)\Bigl)\cr
&=E^X_{\delta_x}\Bigl(e^{-X_{t/2}(V^\lambda_{t/2})}-e^{-X_{t/2}(V^\infty_{t/2})}\Bigr)\cr
&\le E^X_{\delta_x}(X_{t/2}(V^\infty_{t/2}-V^\lambda_{t/2}))\cr
&=\int p(t/2,y-x)(V^\infty(t/2,y)-V^\lambda(t/2,y))\,dy\cr
&\le t^{-1/2}\frac{e^{4/t}C(t/2)}{r-1}\lambda^{1-r},
\end{align}
where we used \eqref{lebintbnd} in the last line. Finally use \eqref{expbnds} again to
obtain the required bound.  
\end{proof}
 The critical term in \eqref{LTdecomp} is $\exp\Bigl(-\int_0^TH_{e^s}(Y_s)\,ds\Bigr)$.  To estimate its decay rate, we introduce
 \begin{align}
Z_T=&Z_T(Y)=\exp\Bigl(\int_0^T F(Y_s)-H_{e^s}(Y_s)\,ds\Bigr)\cr
\label{Zdef} =&\exp\Bigl(\int_0^TV^\infty(1,Y_s)-V^{e^{s/2}}(1,Y_s)\,ds\Bigr)\uparrow Z_\infty(Y)\le \infty\text{ as }T\to\infty.
\end{align}
Let $\lambda_0=\lambda_0(F)\in(\frac{1}{2},1)$ be as in Proposition~\ref{lam0}. Choose $\vep\in(0,2\lambda_0-1)$ and set 
\begin{equation}\label{del1}\delta_{\eqref{del1}}=2\lambda_0-\vep>1.
\end{equation}

\begin{lemma}\label{prel}(a) For all $t>0$ there is a $C(t)$, non-increasing in $t$, such that $\sup_xV^\infty(t.x)-V^\lambda(t,x)\le C(t)\lambda^{1-\delta_{\eqref{del1}}}$ for all $\lambda>0$.

\noindent(b) There is a constant $C$ so that $Z_\infty-Z_T\le Ce^{-T(\delta_{\eqref{del1}}-1)/2}$ for all $T\ge 0$.  In particular $Z_\infty$ is uniformly bounded by some constant $C_Z$. 
\end{lemma}
\begin{proof}
By Lemma~\ref{Fprop}(b) and \eqref{Hdef} we may first choose $K$ and then $T_0$ so that 
\begin{equation}\label{Hconv}
\sup_{|x|\ge K} F(x)<\vep/2,\quad \sup_{s\ge T_0}\sup_{|x|\le K}F(x)-H_{e^s}(x)<\vep/2,
\end{equation}
which in turn implies
\begin{equation}\label{unif}
\sup_{s\ge T_0}\sup_xF(x)-H_{e^s}(x)<\vep/2.
\end{equation}
Now take $h\equiv 1$ in Lemma~\ref{LTequation}, recall that $m$ is the invariant law for $Y$ and use the Markov property at $T_0$ to conclude that for $T\equiv\log(\lambda^2 t)\ge T_0$,
\begin{align}
 E^X_{X_0}&\Bigl(\int e^{-\lambda X(t,x)}X(t,x)dx\Bigr)/X_0(1)\cr
&\le E^Y_m\Bigl(\exp\Bigl(-\int_0^TH_{e^s}(Y_s)\,ds\Bigr)\Bigr)\cr
&\le E^Y_m\Bigl(E^Y_{Y_{T_0}}\Bigl(\exp\Bigl(-\int_0^{T-T_0} H_{e^{s+T_0}}(Y_s)\,ds\Bigr)\Bigr)\Bigr)\cr
&\le E^Y_m\Bigl(\exp\Bigl(-\int_0^{T-T_0}F(Y_s)\,ds\Bigr)\Bigr)\exp((\vep/2)(T-T_0))\quad \text{by \eqref{unif}}\cr
\label{rhobound}&\le(t\lambda^2)^{\vep/2}P^Y_m(\rho_F>T-T_0),
\end{align}
where we recall that $\rho_F$ is the lifetime of the killed Ornstein-Uhlenbeck process $Y^F$.  Now use Theorem~\ref{fens}(d) to see that for $\lambda\ge \lambda'(\vep)$, \eqref{rhobound} is at most
\begin{equation}\nonumber
(t\lambda^2)^{\vep/2}e^{-\lambda_0(T-T_0)}[\theta\int\psi_0dm+c' e^{-(\lambda_1-\lambda_0)(T-T_0)}]\le c\cdot(\sqrt t\lambda)^{-\delta_{\eqref{del1}}},
\end{equation}
for some universal constant $c$.  We now may apply Lemma~\ref{prelgen} to conclude that 
\begin{equation*} 
\sup_xV^\infty(t,x)-V^\lambda(t,x)\le C(t)\lambda^{1-\delta_{\eqref{del1}}},
\end{equation*}
first for $\lambda>\lambda(\vep)$, and then for all $\lambda> 0$, the latter
using \eqref{Vinftybound} and by increasing $C(t)$.  It is easy to use the explicit form for the constant in Lemma~\ref{prelgen} to see that we may take $C(t)$ to be non-increasing in $t$.  

Turning next to (b), we have from \eqref{Zdef} and (a), 
$$Z_T\le \exp\Bigl(C(1)\int_0^\infty e^{-(\delta_{\eqref{del1}}-1)s/2}\,ds\Bigr)\equiv c_0.$$
This bound and (a) imply that
\begin{align*}
Z_\infty-Z_T&\le Z_\infty\Bigl[1-\exp\Bigl(-\int_T^\infty V^\infty(1,Y_s)-V^{e^{s/2}}(1,Y_s)\,ds\Bigr)\Bigr]\cr
&\le c_0C(1)\int_T^\infty e^{-(\delta_{\eqref{del1}}-1)s/2}\,ds\cr
&=\frac{c_0C(1)}{\delta_{\eqref{del1}}-1}e^{-(\delta_{\eqref{del1}}-1)T/2}.
\end{align*}
\end{proof}
Recall the law $P_x^\infty$ from Theorem~\ref{fens}(e). We need a slight extension of the latter result. 
\begin{lemma}\label{condlimit}
For any $\phi:\R\to\R$ bounded measurable,
\[\lim_{T\to\infty}E^Y_x(Z_T\phi(Y_T) |\rho>T)=E_x^\infty(Z_\infty)\int\phi\psi_0dm/\theta.\]
\end{lemma}
\begin{proof} If $T>T_1>0$, then the monotonicity of $T\rightarrow Z_T$ and Lemma~\ref{prel}(b) imply 
that 
\begin{equation}\label{TT1bnd}
E^Y_x(|Z_T-Z_{T_1}| |\phi|(Y_T)\,|\rho>T)\le C\Vert\phi\Vert_\infty e^{-T_1(\delta_{\eqref{del1}}-1)/2}.
\end{equation}
So using this and the bound in Lemma~\ref{prel}(b), it clearly suffices to show that 
for each $T_1>0$,
\begin{equation}\label{T1limi}
\lim_{T\to\infty} E^Y_x(Z_{T_1}\phi(Y_T)|\rho>T)=E_x^\infty(Z_{T_1})\int\phi\psi_0dm/\theta.
\end{equation}
By the Markov property and the eigenfunction expansion in Theorem~\ref{fens}(c), the lefthand side of the above is
\begin{align}
\label{efnexp}\lim_{T\to\infty}E^Y_x(&Z_{T_1}1(\rho>T_1)\int q(T-T_1,Y_{T_1},z)\phi(z)\,dm(z)/P^Y_x(\rho>T)\cr
=&\lim_{T\to\infty}E^Y_x(Z_{T_1}1(\rho>T_1)e^{-\lambda_0(T-T_1)}\int\phi\psi_0\,dm\,\psi_0(Y_{T_1}))/P^Y_x(\rho>T)\cr
&\phantom{\lim_{T\to\infty}E_x(Z_{T_1}1(\rho>T_1)e^{-\lambda_0(T-T_1)}}+\lim_{T\to\infty}\delta(x,T),
\end{align}
where (recall that $Z_\infty$ is uniformly bounded)
\begin{equation*}
|\delta(x,T)|\le C\Vert\phi\Vert_\infty E^Y_x\Bigl(\int\sum_{n=1}^\infty|\psi_n(Y_{T_1})|e^{-\lambda_n(T-T_1)}|\psi_n(z)|\,dm(z)\Bigr)/P^Y_x(\rho>T).
\end{equation*}
Now use the second inequality in \eqref{rboundnew} 
with $\delta=1/8$ to deduce that for $T-T_1\ge s^*(1/8)$, 
\begin{eqnarray*} 
|\delta(x,T)|&\le&C\Vert\phi\Vert_\infty e^{-\lambda_1T}c_\delta E^Y_x\Bigl(e^{Y^2_{T_1}/8}\Bigr)/P^Y_x(\rho>T)\cr
&\le&C'(x)e^{-\lambda_1T}/P^Y_x(\rho>T)\to0\text{ as }T\to\infty,
\end{eqnarray*}
the last convergence by \eqref{rhotail} and \eqref{rbound}.  \eqref{rhotail} also shows that the first term in \eqref{efnexp} is
\begin{align*}\lim_{T\to\infty}E^Y_x&(Z_{T_1}\psi_0(Y_{T_1})1(\rho>T_1))\frac{e^{\lambda_0T_1}}{\psi_0(x)\theta+r(T,x)}\int\phi\psi_0\,dm\cr
&=E^Y_x\Bigl(Z_{T_1}\frac{\psi_0(Y_{T_1})}{\psi_0(x)}1(\rho>T_1)\Bigr)e^{\lambda_0T_1}\int\phi\psi_0\,dm/\theta\quad\text{(by \eqref{rbound})}\cr
&=E_x^\infty(Z_{T_1})\int\phi\psi_0\,dm/\theta.
\end{align*}
This establishes\eqref{T1limi} and so completes the proof.
\end{proof}


\begin{prop} \label{LTlimit} Assume $h\ge 0$ is a bounded Borel function on the line.

\noindent(a) There is a universal constant $C_{\ref{LTlimit}}>0$ such that for any $t>0$,
\begin{align*}\lim_{\lambda\to\infty}& (\lambda^2 t)^{\lambda_0} E^X_{X_0}\Bigl(\int e^{-\lambda X(t,x)}h(x)X(t,x)\,dx\Bigr)\cr
&=C_{\ref{LTlimit}}\int\int h(w_0+\sqrt t z)\exp\Bigl(-t^{-1}\int F(z+t^{-1/2}(w_0-x_0))dX_0(x_0)\Bigr)\cr
&\phantom{\quad\times\int\int h(w_0+\sqrt t z)\exp\Bigl(-t^{-1}\int F}\times\psi_0(z)dm(z)dX_0(w_0).
\end{align*}

\noindent (b) There is a constant $C$ such that for all $\lambda,t>0$,
$$ (\lambda^2 t)^{\lambda_0} E^X_{X_0}\Bigl(\int e^{-\lambda X(t,x)}h(x)X(t,x)\,dx\Bigr)\le C\Vert h\Vert_\infty X_0(1).$$
\end{prop}
\begin{proof} Set $T=\log(\lambda^2 t)$. To simplify the expression obtained in Lemma~\ref{LTequation}, introduce
$$I_T(w_0,y)=\exp\Bigl(-t^{-1}\int H_{e^T}(y+t^{-1/2}(w_0-x_0))\,dX_0(x_0)\Bigr),$$
$$I_\infty(w_0,y)=\exp\Bigl(-t^{-1}\int F(y+t^{-1/2}(w_0-x_0))\,dX_0(x_0)\Bigr),$$
and 
$$\Psi(w_0,x,T)=E^Y_x\Bigl(\exp\Bigl(-\int_0^T F(Y_s)ds\Bigr) Z_T h(w_0+\sqrt t Y_T)I_T(w_0,Y_T)\Bigr).$$
then Lemma~\ref{LTequation} states that for $\lambda^2t\ge 1$,
\begin{align}\label{LTdecomp2}
E^X_{X_0}&\Bigl(\int e^{-\lambda X(t,x)} h(x)X(t,x)\,dx\Bigr)\cr
&=E_0^B\Bigl(\exp\Bigl(-\int_0^1 V^1(u,B(u))\,du\Bigr)\int \Psi(w_0,B_1,T)\,dX_0(w_0)\Bigr).
\end{align}
It follows from \eqref{Hdef} and Lemma~\ref{prel}(a) that 
\begin{equation*} 0\le F(x)-H_u(x)\le c(1)u^{(1-\delta_{\eqref{del1}})/2}\quad\forall u>0, \ x\in\R,
\end{equation*}
and so 
\begin{equation}\label{Irate}
0\le (I_T-I_\infty)(w_0,Y_T)\le \frac{X_0(1)}{t}c(1)e^{T(1-\delta_{\eqref{del1}})/2}=c(1)X_0(1)t^{-(1+\delta_{\eqref{del1}})/2}\lambda^{1-\delta_{\eqref{del1}}}.
\end{equation}
Therefore
\begin{align}\label{psieqn}
\frac{\Psi(w_0,x,T)}{P^Y_x(\rho>T)}&=E^Y_x(Z_T(Y)h(w_0+\sqrt t Y_T)I_T(w_0,Y_T)|\rho>T)\cr
&=E^Y_x(Z_T(Y)h(w_0+\sqrt t Y_T)I_\infty(w_0,Y_T)|\rho>T)+\delta(T,x,w_0),
\end{align}
where by \eqref{Irate},
\begin{equation}\label{deltarate}
|\delta(T,x,w_0)|\le \Vert h\Vert_\infty C_Zc(1)X_0(1)t^{-(1+\delta_{\eqref{del1}})/2} \lambda^{1-\delta_{\eqref{del1}}}\to 0\text{ as }\lambda\to\infty.
\end{equation}
Lemma~\ref{condlimit}, together with \eqref{psieqn} and \eqref{deltarate}, show that
\begin{equation}\label{psilimit}
\lim_{\lambda\to\infty}\frac{\Psi(w_0,x,T)}{P^Y_x(\rho>T)}=E^\infty_x(Z_\infty)\int h(w_0+\sqrt tz)I_\infty(w_0,z)\psi_0(z)dm(z)/\theta,
\end{equation}
and the first line in \eqref{psieqn} implies 
\begin{equation}\label{psibnd}
\frac{|\Psi(w_0,x,T)|}{P^Y_x(\rho>T)}\le C_Z\Vert h\Vert_\infty.
\end{equation}
Apply \eqref{rhotail} and \eqref{rbound} with $\delta=1/8$ to conclude 
\begin{equation}\label{killrate}
|e^{\lambda_0 T}P^Y_x(\rho>T)-\theta\psi_0(x)|\le c_\delta e^{\delta x^2}e^{-(\lambda_1-\lambda_0)T}.
\end{equation}
Returning to \eqref{LTdecomp2} we have for $\lambda^2t\ge 1$ (assumed henceforth)
\begin{align}\label{unifLTbound}
(&\lambda^2 t)^{\lambda_0}E^X_{X_0}\Bigl(\int e^{-\lambda X(t,x)}X(t,x)\,dx\Bigr)\cr
&=E_0^B\Bigl[\exp\Bigl(-\int_0^1V^1(u,B(u))\,du\Bigr)\int\frac{\Psi(w_0,B_1,T)}{P^Y_{B_1}(\rho>T)}dX(w_0) e^{\lambda_0T}P^Y_{B_1}(\rho>T)\Bigr]\cr
&\rightarrow E_0^B\Bigl(\exp\Bigl(-\int_0^1V^1(u,B(u))\,du\Bigr)E^\infty_{B_1}(Z_\infty)\cr
&\phantom{\rightarrow E_0^B\Bigl(\exp\Bigl(-\int_0^1}\times\int h(w_0+\sqrt t z)I_\infty(w_0,z)\psi_0(z)dm(z)dX_0(w_0)\psi_0(B_1)\Bigr)
\end{align}
as $\lambda\to\infty$, where \eqref{psilimit} is used in the last and Dominated Convergence may be applied thanks to \eqref{psibnd}, \eqref{killrate} and \eqref{psi0bnd}. This gives (a) with 
\begin{equation}\label{CForm}C_{\ref{LTlimit}}=E_0^B\Bigl(\exp\Bigl(-\int_0^1 V^1(u,B_u)\,du\Bigr) E_{B_1}^\infty(Z_\infty(Y))\psi_0(B_1)\Bigr).
\end{equation}
For (b), if $\lambda^2 t\ge 1$, use \eqref{psibnd}, \eqref{killrate} and \eqref{psi0bnd} to bound the second line in the display \eqref{unifLTbound} by
$$c\Vert h\Vert_\infty X_0(1)E_0^B(e^{B(1)^2/8})\le c\Vert  h\Vert_\infty X_0(1).$$
If $\lambda^2t<1$, the expression to be bounded is at most $\Vert h\Vert_\infty E^X_{X_0}(X_t(1))=\Vert h\Vert_\infty X_0(1)$.
\end{proof}

Here is the promised refinement of Lemma~\ref{prel}(a) giving the exact rate of convergence in Proposition~\ref{kamin}(a).
\begin{prop}\label{Vrate}
\noindent (a) There is a constant $\bar C_{\ref{Vrate}}$ such that
$$\sup_x V^\infty(t,x)-V^\lambda(t,x)\le \bar C_{\ref{Vrate}}t^{-\frac{1}{2}-\lambda_0}\lambda^{-(2\lambda_0-1)}\ \forall\lambda>0.$$
\noindent(b) For any $K\ge 1$ there is a $\underline C_{\ref{Vrate}}(K)>0$ which is non-increasing in $K$ such that for any $t>0$,
\begin{equation}\label{Vlb}\inf_{|x|\le K\sqrt t} V^\infty(t,x)-V^\lambda(t,x)\ge \underline C_{\ref{Vrate}}(K)t^{-\frac{1}{2}-\lambda_0}\lambda^{-(2\lambda_0-1)}\ \forall\lambda\ge t^{-1/2},
\end{equation}
and
 \begin{equation}\label{Vlb2}\inf_{|x|\le K\sqrt t} V^\infty(t,x)-V^\lambda(t,x)\ge \underline C_{\ref{Vrate}}(K)(t^{-1}\wedge t^{-\frac{1}{2}-\lambda_0})\lambda^{-(2\lambda_0-1)}\ \forall\lambda\ge1.
 \end{equation}
\end{prop}
\begin{proof} (a) Apply Proposition~\ref{LTlimit}(b) and Lemma~\ref{prelgen} to see that for $\lambda^2 t\ge 1$, 
\begin{align*}
\sup_x V^\infty(t,x)-V^\lambda(t,x)&\le e^{6/t}t^{-1/2-\lambda_0}C(2\lambda_0-1)^{-1}\lambda^{-(2\lambda_0-1)}\cr
&=C'e^{6/t}t^{-1}(\sqrt t \lambda)^{-(2\lambda_0-1)}.
\end{align*}
For $\lambda^2 t<1$, by \eqref{Vinftybound} the lefthand side of the above is at most 
$$\sup_xV^\infty(t,x)\le 2t^{-1}\le 2e^{6/t}t^{-1}(\sqrt t \lambda)^{-(2\lambda_0-1)}.$$
This proves (a) but with an additional factor of $e^{6/t}$.  This can be removed by applying the scaling result \eqref{pdescale} and the above bound for $t=1$
to conclude that for all $\lambda>0$,
\begin{equation}\label{scale2}
V^\infty(t,x)-V^\lambda(t,x)=t^{-1}(V^\infty-V^{\sqrt t\lambda})(1,x/\sqrt t)\le \bar C t^{-1}(\sqrt t\lambda)^{-(2\lambda_0-1)}.
\end{equation}
\noindent(b) Set $h=1_{[-K,K]}$ in Proposition~\ref{LTlimit} with $X_0=\delta_0$ and argue as in the first line of \eqref{Vderiv} to see that for $\lambda\ge\underline\lambda(t,K)$ and a universal positive constant $c$,
\begin{align}\label{intVbnd}
-\frac{d}{d\lambda}\int_{-K}^K e^{-V^\lambda(t,x)}-&e^{-V^\infty(t,x)}\,dx
=E_{\delta_0}\Bigl(\int e^{-\lambda X(t,x)}h(x) X(t,x)\,dx\Bigr)\cr
&\ge c\int_{-K/\sqrt t}^{K/\sqrt t} \exp(-t^{-1}F(z))\psi_0(z)dm(z)t^{-\lambda_0}\lambda^{-2\lambda_0}\cr
&=c_0(t,K)\lambda^{-2\lambda_0}.
\end{align}
Integrate out $\lambda$ to conclude that for $\lambda\ge \underline\lambda(t,K)$, 
\begin{align}\label{Vint1}\int_{-K}^KV^\infty(t,x)-V^\lambda(t,x)\,dx&\ge \int_{-K}^Ke^{-V^\lambda(t,x)}-e^{-V^\infty(t,x)}\,dx\cr
&\ge \frac{c_0(t,K)}{2\lambda_0-1}\lambda^{1-2\lambda_0}.
\end{align}
If $0<\lambda<\underline\lambda(t,K)$, then
\begin{equation}\label{Vint2}
\int_{-K}^K (V^\infty-V^\lambda)(t,x)\,dx\ge \int_{-K}^K (V^\infty-V^{\underline\lambda(t,K)})(t,x)\,dx\equiv c_1(t,K)>0.
\end{equation}
The last inequality holds since the first line of \eqref{Vincrement} shows strict positivity of $V^\infty(t,x) -V^\lambda(t,x)$ for all $t>0$ and $x$.  \eqref{Vint1} and \eqref{Vint2} imply that for some $c_2(t,K)>0$,
\begin{equation}\label{Vint3}\int_{-K}^K (V^\infty-V^\lambda)(t,x)\,dx\ge c_2(t,K)\lambda^{1-2\lambda_0}\quad\forall\lambda\ge 1.
\end{equation}

From the third line of \eqref{Vincrement} we have 
\begin{align*}
e^{-V^\lambda(t,x)}-e^{-V^\infty(t,x)}&=E_{\delta_x}\Bigl(e^{-X_{t/2}(V^\lambda_{t/2})}-e^{-X_{t/2}(V^\infty_{t/2})}\Bigr)\cr
&\ge  E_{\delta_x}\Bigl(e^{-X_{t/2}(V^\infty_{t/2})}X_{t/2}(V^\infty_{t/2}-V^\lambda_{t/2})\Bigr)\cr
&\ge  E_{\delta_x}\Bigl(e^{-(4/t)X_{t/2}(1)}X_{t/2}(V^\infty_{t/2}-V^\lambda_{t/2})\Bigr)\quad\text{(by\eqref{Vinftybound})}.
\end{align*}
So if $r=r(t)=4/t$ and $G=G(t)=V^\infty_{t/2}-V^\lambda_{t/2}$, we seek a lower bound on 
\begin{equation}\label{Camp}E^X_{\delta_x}(e^{-rX_{t/2}(1)}X_{t/2}(G)).
\end{equation}
If $V_t(\phi)$ denotes the nonlinear semigroup associated with $X$ (see Section II.5 of \cite{per02}), we may use the Campbell measure formula for $X_t$ as in \eqref{campbell} along with $V_s(r)=\frac{2r}{2+rs}$ (ie. the non-linear semigroup with $\phi\equiv r$ constant) to see that if $W$ is a Brownian motion starting at $x$, then for $|x|\le K$, \eqref{Camp}
equals
\begin{align}\label{scaleVbnd}
E^W_x&\times E^X_{\delta_x}\Bigl(e^{-rX_{t/2}(1)}\exp\Bigl(-\int_0^{t/2}\frac{2r}{2+rs}\,ds\Bigr)G(W_{t/2})\Bigr)\cr
&=\exp\Bigl(\frac{-2r}{2+(rt/2)}\Bigr)\Bigl(1+\frac{rt}{4}\Bigr)^{-2}\int p_{t/2}(y-x)G(y)\,dy\cr
&\ge c(t,K)\int_{-K}^KG(y)\,dy\cr
&\ge c'(t,K)\lambda^{-(2\lambda_0-1)},\quad\text{for all }\lambda\ge1,
\end{align}
the last by \eqref{Vint3}. Now use the first inequality in \eqref{expbnds} and the above to derive
\[\inf_{|x|\le K} V^\infty(t,x)-V^\lambda(t,x)\ge c'(t,K)\lambda^{-(2\lambda_0-1)}\text{ for all }\lambda\ge 1,\]
and where we may assume $c'(t,K)>0$ is non-increasing in $K$ for each $t$. 

The scaling relation in \eqref{scale2} and the above bound for $t=1$ shows that for all $\lambda\ge t^{-1/2}$, 
\begin{align}\label{scalebnd} \inf_{|x|\le K\sqrt t}V^\infty(t,x)-V^\lambda(t,x)&=\inf_{|x|\le K}t^{-1}(V^\infty-V^{\sqrt t\lambda})(1,x)\cr
&\ge c'(1,K)t^{-\frac{1}{2}-\lambda_0}\lambda^{-(2\lambda_0-1)}.
\end{align}
This gives \eqref{Vlb} and it remains to prove \eqref{Vlb2}. By \eqref{Vlb} we may assume $t<1$ and $1\le \lambda\le t^{-1/2}$.  Then $c(K)\equiv\inf_{|x|\le K}V^\infty(1,x)-V^1(1,x)>0$, where the last inequality holds by the strict positivity of the difference for each $x$ as noted above.  By scaling, as in the above display, the lefthand side of \eqref{scalebnd} is at least $t^{-1}c(K)$ which implies \eqref{Vlb2}.
\end{proof}

The following Tauberian Theorem is implicit in Theorem 1 of \cite{HS85} (see especially p. 350).  The explicit constants below are not given there but follow from an elementary argument given in Appendix~\ref{Tauberian} below. 

\begin{lemma}\label{Taub}
Let $U$ be the distribution function of a sub-probability on $(0,\infty)$, set $\hat U(\lambda)=\int_0^\infty e^{-\lambda x}dU(x)$ and let $p>0$. 

\noindent(a) Assume for some $C_2>0$, 
\begin{equation}\label{LTUbound}\hat U(\lambda)\le C_2\lambda^{-p}\text{ for all }\lambda>0.\end{equation}
  Then $U(a)\le eC_2a^p$ for all $a>0$.

\noindent(b) Assume \eqref{LTUbound} and for some $C_1>0$, $\underline\lambda\ge 0$, 
\begin{equation}\label{LTLbound}
\hat U(\lambda)\ge C_1 \lambda^{-p}\text{ for all }\lambda> \underline\lambda.
\end{equation}
If $d_1=\frac{C_1}{2}\Bigl(2\log\Bigl(\Bigl(\frac{2p}{e}\Bigr)^p\frac{4eC_2}{C_1}\Bigr)\vee 2p\vee\underline\lambda\Bigr)^{-p}$, then 
\begin{equation}\label{ULbound}
U(a)\ge d_1a^p\quad\text{ for all }a\in[0,1].
\end{equation}
In particular if $p\le 1$ and $\underline\lambda\le 4$, then 
\begin{equation}\label{ULboundb}
U(a)\ge \frac{C_1}{4}\Bigl(\log\Bigl(\frac{4eC_2}{C_1}\Bigr)\Bigr)^{-1}a^p\quad\text{ for all }a\in [0,1].
\end{equation}
\end{lemma}

\begin{theorem} \label{SBMdensity}Let $X(t,x)$ be the density of super-Brownian 
motion satisfying \eqref{SPDE} with finite initial measure $X_0$.

\noindent(a) $P^X_{X_0}(0<X(t,x)\le a)\le e\bar C_{\ref{Vrate}}X_0(1)t^{-(1/2)-\lambda_0}a^{2\lambda_0-1}$ $\forall a,t>0, x\in\R$.

\noindent(b) For all $K\ge 1$ there is a $\underline C_{\ref{SBMdensity}}(K)>0$ such that if $X_0(1)\le Kt$ and $X_0([x-K\sqrt t,x+K\sqrt t])/X_0(1)\ge K^{-2}$, then
\begin{equation}\label{denslbound}
P^X_{X_0}(0<X(t,x)\le a)\ge\underline C_{\ref{SBMdensity}}(K)X_0(1) t^{-(1/2)-\lambda_0}a^{2\lambda_0-1}\quad\forall\, 0\le a\le \sqrt t.
\end{equation}
In particular if $|x-x_0|\le K\sqrt t$ and $t\ge K^{-1}$, then 
\begin{equation}
P^X_{\delta_{x_0}}(0<X(t,x)\le a)\ge\underline C_{\ref{SBMdensity}}(K) t^{-(1/2)-\lambda_0}a^{2\lambda_0-1}\quad\forall\, 0\le a\le \sqrt t,
\end{equation}
and if $X_0(1)\le K$, $X_0([-K,K])\ge K^{-1}$, $t\ge K^{-1}$, and $|x|\le K$, then
\begin{equation}\label{denslbound2}
P^X_{X_0}(0<X(t,x)\le a)\ge\underline C_{\ref{SBMdensity}}(K^2)K^{-1}t^{-(1/2)-\lambda_0}a^{2\lambda_0-1}\quad\forall\, 0\le a\le \sqrt t.
\end{equation}
\end{theorem} 
\begin{proof} We will apply Lemma~\ref{Taub} to $U(a)=P(0<X(t,x)\le a)$.  By \eqref{LLEb}, \eqref{Vinfform},  and translation invariance, 
\begin{align}\label{LTFormula}
\hat U(\lambda)&=E^X_{X_0}(e^{-\lambda X(t,x)}1(X(t,x)>0))\cr
&=\exp\Bigl(-\int V^\lambda(t,y-x)dX_0(y)\Bigr)-\exp\Bigl(-\int V^\infty(t,y-x)dX_0(y)\Bigr).
\end{align}
(a) Use Proposition~\ref{Vrate}(a) to see that for $\lambda>0$,
\begin{equation}\label{hatUbound}
\hat U(\lambda)\le \int(V^\infty-V^\lambda)(t,y-x)dX_0(y)\le \bar C_{\ref{Vrate}} X_0(1)t^{-(1/2)-\lambda_0}\lambda^{-(2\lambda_0-1)},
\end{equation}
and so (a) is immediate from Lemma~\ref{Taub}(a).

\noindent (b) Consider first $t=1$. By \eqref{LTFormula}, \eqref{Vinftybound} and \eqref{Vlb} for $\lambda\ge 1$, 
\begin{align*}
\hat U(\lambda)&\ge \exp\Bigl(-\int V^\infty(1,y-x)dX_0(y)\Bigr)\int(V^\infty-V^\lambda)(1,y-x)dX_0(y)\cr
&\ge \exp(-2X_0(1))\underline C_{\ref{Vrate}}(K)X_0([x-K,x+K])\lambda^{-(2\lambda_0-1)}\cr
&\equiv C_1(K,x,X_0)\lambda^{-(2\lambda_0-1)}.
\end{align*}
So by this and \eqref{hatUbound} we may use \eqref{ULboundb} in Lemma~\ref{Taub} (with $p=2\lambda_0-1<1$) to see that
\begin{equation*}
P^X_{X_0}(0<X(1,x)\le a)\ge \frac{C_1(K,x,X_0)}{4}\Bigl(\log\Bigl(\frac{4e\bar C_{\ref{Vrate}} X_0(1)}{C_1(K,x,X_0)}\Bigr)\Bigr)^{-1}a^{2\lambda_0-1}\quad\forall a\in[0,1]. 
\end{equation*}
For general $t$ we may use the scaling relation \eqref{scale2} and \eqref{LTFormula} to see that if $X_0^t(A)=t^{-1}X_0(\sqrt t A)$, then by the above 
for $0\le a\le \sqrt t$,
\begin{align}\label{LB1}
P&^X_{X_0}(0<X(t,x)\le a)\cr
&=P^X_{X_0^t}(0<X(1,x/\sqrt t)\le a/\sqrt t)\cr
&\ge \frac{C_1(K,x/\sqrt t,X_0^t)}{4}\Bigl(\log\Bigl(\frac{4e\bar C_{\ref{Vrate}} X_0(1)/t}{C_1(K,x/\sqrt t,X^t_0)}\Bigr)\Bigr)^{-1}(a/\sqrt t)^{2\lambda_0-1}.
\end{align}
A simple calculation, using the definition of $C_1$, shows that there is a $C_0(K)>0$ so that if $X_0$ is as in (b), then 
\[C_1(K,x/\sqrt t,X_0^t)\ge C_0(K) X_0(1)/t,\]
and so 
\[\log\Bigl(\frac{4e\bar C_{\ref{Vrate}} X_0(1)/t}{C_1(K,x,X_0)}\Bigr)\le C_3(K).\]
Use the above in \eqref{LB1} to conclude that 
\begin{equation*}
P^X_{X_0}(0<X(t,x)\le a)\ge C(K)X_0(1)t^{-1}(a/\sqrt t)^{2\lambda_0-1}.
\end{equation*}
This proves \eqref{denslbound}, and the last two assertions follow by elementary reasoning.  (The last follows first for $K\ge 4$, and hence for all $K\ge 1$.)
\end{proof}

\section{Proof of Theorem~\ref{dimBZ}}
\subsection {Lower Bound on the Hausdorff Dimension}\label{lbondim}
\setcounter{equation}{0}
Recall that $X$ is as in \eqref{SPDE} and the boundary of the zero set is
\begin{equation}\label{BZdef2}
BZ_t=\partial(\{x:X(t,x)=0\})=\{x:X(t,x)=0, \forall \delta>0\, X_t((x-\delta,x+\delta))>0\}.
\end{equation}
The key step in our lower bound on the Hausdorff dimension of the boundary of the zero set, $\text{dim}(BZ_t)$, is the following second moment  bound. 
The bound will depend on on a diffusion parameter $\sigma_0^2>1$, whose exact value is not important (but $\sigma_0^2=6$ will work).
For a finite initial measure $X_0$ and $t>0$, define $X_0p_u(x)=\int p_u(x-w_0)X_0(dw_0)$ and
\begin{align}\label{hdef}
h_{t,X_0}(z_1,z_2)&=t^{-2\lambda_0}\prod_{i=1}^2X_0p_{\sigma_0^2 t}(z_i)\cr
&\quad+t^{-\lambda_0}\int_0^t (t-s)^{-\lambda_0}
p_{8\sigma_0^2(t-s)}(z_1-z_2)X_0p_{4\sigma_0^2 t}(z_1))\,ds.
\end{align}
\begin{prop}\label{secmombnd} 
There is a constant $C_{\ref{secmombnd}}$ such that for all $\lambda^2\ge (9/t)$, and all $z_1,z_2$, 

\noindent (a) $\lambda^{4\lambda_0}E^X_{X_0}(X_t(z_1)X_t(z_2)e^{-\lambda X_t(z_1)-\lambda X_t(z_2)})\le C_{\ref{secmombnd}}h_{t,X_0}(z_1,z_2)$.
\medskip

\noindent(b) $h_{t,X_0}(z_1,z_2)\le C_{\ref{secmombnd}}(t^{-\lambda_0-1/2}X_0(1)|z_1-z_2|^{1-2\lambda_0}+(t^{-\lambda_0-1/2}X_0(1))^2)$.
\end{prop}

We will prove this result in Section~\ref{5.1} below, but first show how it can be used to obtain lower bounds on $\text{dim}(BZ_t)$. The lack of symmetry between $z_1$ and $z_2$ in the definition of $h_{t,X_0}$ indicates that our 
bound is not optimal, but it is the negative power along the diagonal which will be important for us, and our results suggest that this is optimal.

If we introduce random measures 
\begin{equation}\label{LTdef} L^\lambda_t(\phi)=(\lambda^2 t)^{\lambda_0}
\int \phi(x)X_t(x)e^{-\lambda X(t,x)}\,dx,\end{equation}
 then Proposition~\ref{LTlimit} shows that for some finite measure $\ell_t$ and any bounded Borel $\phi$, 
\[\lim_{\lambda\to\infty}E(L^\lambda_t(\phi))=\ell_t(\phi),\]
and the above result shows that $E(L^\lambda_t(1)^2)$ remains bounded as $\lambda\to\infty$. 
\medskip

\noindent{\bf Conjecture.} There is a random finite non-trivial measure $L_t$ on $\R$ such that for any bounded continuous $\phi$, 
\[L^\lambda_t(\phi)\to L_t(\phi)\text{ in }L^2\text{ as }\lambda\to\infty.\]
 Assuming this, it is then not hard to show that for some sequence $\lambda_n\uparrow\infty$,
$L^{\lambda_n}_t$ approaches $L^\lambda_t$ weakly on the space of measures  a.s. and that $L_t$ is supported by $BZ_t$.  We further conjecture that $L_t(1)>0$ a.s. on $\{X_t(1)>0\}$. 

\medskip

If $g_\beta(r)=r^{-\beta}$ for $\beta>0$, and $\mu$ is a finite measure on $\R$, and $A$ is an analytic subset of $\R$, let
\[\langle\mu,\mu\rangle_{g_\beta}=\int\int g_\beta(|x-y|)\,d\mu(x)d\mu(y),\]
\[I(g_\beta)(A)=\inf\{\langle\mu,\mu\rangle_{g_\beta}:\mu \text{ a probability supported by }A\}.\]
Then the $g_\beta$-capacity of $A$ is $C(g_\beta)(A)=I(g_\beta)(A)^{-1}$ (see, e.g. \cite{Hawkes79}, Section 3).  

\begin{theorem}\label{capcond}
For every $K\ge 1$ there is a positive constant $C_{\ref{capcond}}(K)$, non-increasing in $K$, so that for any analytic subset $A$ of $[-K,K]$, initial measure, $X_0$, satisfying $X_0(1)\le K$ and $X_0([-K,K])\ge 1/K$, and $t\in[K^{-1},K]$,
\begin{equation*} 
P^X_{X_0}(A\cap BZ_t\neq\emptyset)\ge C_{\ref{capcond}}(K)C(g_{2\lambda_0-1})(A).
\end{equation*}
\end{theorem}
\begin{proof}  Let $0<\delta_0<e^{-1}$, and let $0<k_1<1<k_2$ be the solutions of $\delta_0=k_ie^{-k_i}$. We will choose $\delta_0$ small enough below, noting that as $\delta\downarrow 0$, $k_1(\delta)\downarrow 0$ and $k_2(\delta)\uparrow\infty$. We approximate $BZ_t$ by
\begin{equation*}
BZ(\vep)\equiv\{x: X(t,x)e^{-X(t,x)/\vep}\ge \delta_0\vep\}=\{x:k_1\vep\le X(t,x)\le k_2\vep\},
\end{equation*}
where the second equality is by an elementary calculus argument and we have suppressed dependence on $t>0$.  Now fix $K\ge 1$ and assume $X_0$ and $t$ are is in the statement of the Theorem.  Let $F$ be a compact subset of $[-K,K]$. If $I(A)=I(g_{2\lambda_0-1})(A)$ and $C(A)=C(g_{2\lambda_0-1})(A)$, we may choose $\{x_i^N:1\le i\le N\}\subset F$ so that (suppressing the superscript $N$) as $N\to\infty$,
\begin{equation}\label{capapprox}
I_N\equiv \frac{1}{N(N-1)}\sum_i\sum_{j\neq i}|x_i-x_j|^{-(2\lambda_0-1}\to I(F)=1/C(F).
\end{equation}
By Theorem~\ref{SBMdensity} there are constants $\bar C(K)>\underline C(K)>0$ so that 
for \break
$0<\vep<\vep_0(K)$,
\begin{align*}
P&(k_1\vep\le X(t,x_i)\le k_2\vep)\cr
&\ge [\underline C(K)k_2^{2\lambda_0-1}-\bar C(K)k_1^{2\lambda_0-1}]\vep^{2\lambda_0-1}\cr
&\ge \vep^{2\lambda_0-1},
\end{align*}
where in the last line we have chosen $\delta_0=\delta_0(K)$ sufficiently small so that $k_2$ is very large and $k_1$ is close to $0$. 
Therefore by inclusion-exclusion and Proposition~\ref{secmombnd}, for $\vep<\vep_0(K)$,
\begin{align*}
P^X_{X_0}(F\cap BZ(\vep)\neq \emptyset)&\ge \sum_{i=1}^N P^X_{X_0}(x_j\in BZ(\vep))-{\sum\sum}_{i\neq j}P^X_{X_0}(x_i,x_j\in BZ(\vep))\cr
&\ge N\vep^{2\lambda_0-1}-\sum_{i\neq j}\frac{E^X_{X_0}(X(t,x_i)X(t,x_j)e^{-X(t,x_i)/\vep-X(t,x_j)/\vep})}{(\delta_0\vep)^2}\cr
&\ge N\vep^{2\lambda_0-1}-c'(K)\sum_{i\neq j}(1+|x_i-x_j|^{1-2\lambda_0})\vep^{4\lambda_0-2}\cr
&\ge N\vep^{2\lambda_0-1}-c(K)[N\vep^{2\lambda_0-1}]^2I_N.
\end{align*}
Now choose $\vep_N\rightarrow0$ so that $N\vep_N^{2\lambda_0-1}=\frac{1}{2I_Nc(K)}$. Therefore
\begin{equation*}
P^X_{X_0}(F\cap BZ(\vep_N)\neq \emptyset)\ge \frac{1}{4c(K)I_N}\to(4c(K))^{-1}C(F)\text{ as }N\to\infty.
\end{equation*}
This implies that 
\[P^X_{X_0}(F\cap BZ(\vep_N)\neq \emptyset\text{ infinitely often })\ge (4c(K))^{-1}C(F).\]
An elementary argument shows that the event on the left-hand side implies that $F\cap BZ_t\neq\emptyset$ and so the proof is complete for $A=F$ compact.  Use the inner regularity of capacity to now extend the result to analytic subsets of $[-K,K]$. 
\end{proof}

\begin{cor}\label{polbz} Let $A$ be an analytic set such that $C(g_{2\lambda_0-1})(A)>0$. Then for any non-zero $X_0$ and any $t>0$, $P^X_{X_0}(A\cap BZ_t\neq\emptyset)>0$.
\end{cor}
\begin{proof} Choose $K$ large enough so that $C(g_{2\lambda_0-1})(A\cap [-K,K])>0$ (inner regularity of capacity), $X_0([-K,K])\ge 1/K$, and $K>t\vee t^{-1}\vee X_0(1)$.  The result is then immediate from the above theorem.
\end{proof}

\begin{lemma}\label{sublem} Let $\alpha=2\lambda_0-1$ and let $Z$ be the subordinator starting at zero with L\'evy measure $\nu$ where 
\[H(x)=\nu([x,\infty))=x^{-\alpha}(\log((1/x)+1))^2.\]
Then
\begin{equation}\label{capsub} C(g_\alpha)(\{Z_s:s\in(0,1)\})>0\text{ a.s. }
\end{equation}
and
\begin{align}\label{subpolar}
&\text{any analytic set } A \text{ satisfying }dim(A)<2-2\lambda_0\text{ is polar for }Z,\cr
&\qquad\qquad\text{ i.e., }P(Z_t\in A\text{ for some }t>0)=0.
\end{align}
\end{lemma}
\begin{proof} If 
\[g(\lambda)=\int_0^\infty 1-e^{-\lambda u}\,d\nu(u)=\lambda\int_0^\infty H(u)e^{-\lambda u}\,du,\]
then Karamata's Abelian-Tauberian theorem and a short calculation shows that for some $c>0$,
$\lim_{\lambda\to\infty}g(\lambda)/(\lambda^\alpha(\log\lambda)^2)=c$.  If $$f(r)=r^\alpha(\log(1/r))^{-2}(\log\log(1/r))^{1-\alpha}$$ and $f-m(A)$ is the $f$-Hausdorff measure of $A$, then \cite{FP71} (see also Lemma~2.1 of \cite{Hawkes75}) shows that 
$$f-m(\{Z_s:s\le t\})=c_Ht\text{ for some positive }c_H.$$  
By \cite{Taylor61}, this implies \eqref{capsub}.  Note that $\lim_{t\to0+}\frac{tg(1/t)}{t^{1-\alpha}\log(1/t)^2}=c$, 
 and so by Theorem~4.4(iii) of \cite{Hawkes75}, \eqref{subpolar} also holds.  (A short calculus argument shows that $(\log H)''\ge 0$ and so by Theorem~2.1 of \cite{Hawkes75}, the condition~B of Theorem~4.4(iii) is valid.)
\end{proof}

Recall that $\text{dim}(B)$ is the Hausdorff dimension of a set $B\subset\R$.
\begin{theorem}\label{dimlb} If $X_0\neq 0$ and $t>0$, then $P^X_{X_0}(\text{dim}(BZ_t)\ge 2-2\lambda_0)>0$. 
\end{theorem}
\begin{proof} Let $Z_t$ be as in Lemma~\ref{sublem} and set $F=\{Z_s:s\in(0,1)\}$.  We assume $Z$ is independent of the super-Brownian motion $X$ and so work on the product space $(\Omega,\mathcal{F},P)=(\Omega_X,\F_X,P^X_{X_0})\times(\Omega_Z,\F_Z,P^Z_0)$. By \eqref{capsub} and Corollary~\ref{polbz} we have $P(BZ_t(\omega_1)\cap F(\omega_2)\neq \emptyset)>0$.  This implies that, 
$$P^X_{X_0}(\{\omega_1:P^Z_0(\{\omega_2:F(\omega_2)\cap BZ_t(\omega_1)\neq\emptyset\})>0\})>0.$$  By \eqref{subpolar}, this in turn implies that $P^X_{X_0}(\text{dim}(BZ_t)\ge 2-2\lambda_0)>0$, as required.
\end{proof}

\subsection {Upper Bound on the Hausdorff Dimension}\label{dimubound}

We begin with a classical result on the modulus of continuity of the density $X(t,x),\ t>0$, of super-Brownian motion, the solution of \eqref{SPDE}, where the initial condition is an arbitrary finite measure $X_0$. 
\medskip

\noindent{\bf Notation.} If $(t_i,x_i)\in\R_+\times\R$ for $i=1,2$, let $$d((t_1,x_1),(t_2,x_2))=\sqrt{|t_1-t_2|}+|x_1-x_2|.$$

The following is an easy consequence of Theorem II.4.2 (and its proof) of \cite{per02} and standard consequences of Kolmogorov's continuity criteria.  (One should use the decomposition (III.4.11) in \cite{per02} for $t\ge2t_0$.)

\begin{prop}\label{classmod} If $\xi\in(0,1/2)$, then for any $K\in\N$ there is a $\rho(K,\xi,\omega)>0$ a.s. such that 
\begin{align}\label{classmod2} &\forall(t,x)\in[K^{-1},K]\times[-K,K], \,d((t,x),(t',x'))\le \rho\cr
&\text{ implies }
|X(t',x')-X(t,x)|\le d((t',x'),(t,x))^\xi.
\end{align}
Moreover there is a $\delta_{\ref{classmod}}>0$, depending only on $\xi$, and a constant $C(X_0(1),K,\xi)$ so that 
\begin{equation}\label{phiprob}
P^X_{X_0}(\rho(K,\xi)\le r)\le C(K,X_0(1),\xi)r^{\delta_{\ref{classmod}}}\text{ for all }r>0. \end{equation}
\end{prop}


Near the zero set $Z=\{(t,x):X(t,x)=0\}$ one can improve the above modulus since the noise term in \eqref{SPDE} will be mollified.  This idea plays a central role in the pathwise uniqueness arguments in \cite{mps06} and \cite{mp11}. The following result can be derived using the same proof as that of Theorem~2.3 in \cite{mp11} (see also Corollary~4.2 of \cite{mps06}).  There are a few minor changes as these references study the difference of two solutions as opposed to the solution itself.  The minor changes that are required are outlined in Appendix ~\ref{impmoda}. 

\begin{theorem}\label{improvedmod} If $\xi\in(0,1)$, then for any $K\in\N$ there is a $\rho_{\ref{improvedmod}}(K,\xi,\omega)>0$ a.s. such that 
\begin{align}\label{classmod2} &\forall(t,x)\in Z \text{ such that } t\ge K^{-1}, d((t,x),(t',x'))\le \rho_{\ref{improvedmod}}\cr
&\text{ implies }
X(t',x')\le d((t',x'),(t,x))^\xi.
\end{align}
\end{theorem}
In order to get a good cover of $BZ_t$ we need a version of our low density bound
Theorem~\ref{SBMdensity}(a) for small intervals.  (Set $M=1$ and $\vep=a$ in the following result to compare.)

\begin{theorem}\label{lowdensint} There is a $C_{\ref{lowdensint}}>0$ so that for all $t>0$, $M\ge 1$, $0<\vep\le \sqrt t$, $x\in \R$, and $X_0$, 
\[P_{X_0}(0<X_t([x,x+\vep])\le \vep^2 M)\le C_{\ref{lowdensint}} X_0(1)t^{-(1/2)-\lambda_0}M^{19}\vep^{2\lambda_0-1}.\]
\end{theorem}
This will be proved in Section~\ref{5.8} below. We now show how it gives an upper bound on $\text{dim}(BZ_t)$.

\begin{theorem}\label{dimubnd} For all $t>0$, $\text{dim}(BZ_t)\le 2-2\lambda_0$ $P^X_{X_0}$-a.s.
\end{theorem}
\begin{proof} By scaling we may take $t=1$.  By translation invariance, it suffices to show 
\begin{equation}\label{localdim}
\text{dim}(BZ_1\cap[0,1])\le 2-2\lambda_0\ P^X_{X_0}-\text{a.s.}
\end{equation}
Fix $\delta>0$ and choose $\xi\in(0,1)$ so that $19(1-\xi)<\delta$.  Let $\rho_{\ref{improvedmod}}(\xi,\omega)$ be as in Theorem~\ref{improvedmod} with $K=2$.
Then by Theorems~\ref{lowdensint} and \ref{improvedmod} for any $x\in[0,1]$ and $0<\vep\le 1$, 
\begin{align*}
P^X_{X_0}([x,x+\vep]\cap BZ_1\neq\emptyset, 3\vep<\rho_{\ref{improvedmod}})&\le P^X_{X_0}(0<X_1((x-\vep,x+2\vep))\le 6\vep^{\xi+1})\cr
&\le CX_0(1)\vep^{2\lambda_0-1+19(\xi-1)}\cr
&\le CX_0(1)\vep^{2\lambda_0-1-\delta}.
\end{align*}
A standard covering argument using intervals of the form $[i/N,(i+1)/N]$ now gives 
$\text{dim}(BZ_1\cap[0,1])\le 2-2\lambda_0+\delta$ and \eqref{localdim} follows.
\end{proof}

Theorem~\ref{dimBZ} is now immediate from Theorems~\ref{dimlb} and \ref{dimubnd}.
\section{Proof of Proposition~\ref{secmombnd}}
\setcounter{equation}{0}
We now consider the proof of Proposition~\ref{secmombnd}. 
As before, $Y$ is the Ornstein-Uhlenbeck process starting with law $\mu$ under $P^Y_\mu$ and we enlarge
this space to include an independent random variable $W_0$ with ``law" $X_0$.  The same convention is in place on the space carrying a standard Brownian motion starting at $0$ under $P^B_0$, and $P'_t$ denotes the Brownian semigroup.  We fix
a pair of bounded non-negative continuous functions on the line, $\phi_i$, $i=1,2$, and
 define
\begin{equation}\label{psilamdef}
\psi^\lambda(u,x)=E^B_0\Bigl(\phi_2(\lambda^{-1}B_{u\lambda^2}+x)\exp\Bigl\{-\int_0^{\lambda^2 u}V^1(r,B_r)\,dr\Bigr\}\Bigr).
\end{equation}
\begin{lemma}\label{secmom1} There is a constant $C_{\ref{secmom1}}$ so that for all $\lambda_i\ge 0$, $\lambda_i^2t\ge 1$, if $T^i=\log(\lambda_i^2t)$, then
\begin{align}\label{secmom1bnd}E^X_{X_0}\Bigl(&\int\int\phi_1(x_1)\phi(x_2)\exp(-\lambda_1 X_t(x_1)-\lambda_2 X_t(x_2))dX_t(x_1)dX_t(x_2)\Bigr)\cr
\le& C_{\ref{secmom1}}\prod_{i=1}^2E^Y_m(\phi_i(W_0+\sqrt t Y_{T^i})1(\rho>T^i))\cr
&\ +\int_0^tE_0^B\Bigl(\int \phi_1(W_0+B_t)\exp\Bigl(-\int_0^tV^{\lambda_1}_{t-r}(B_t-B_r)\,dr\Bigr)\cr
&\phantom{\le \int_0^tE_0^B\Bigl(\int \phi_1(W_0+B_t)\exp}\times\psi^{\lambda_2}(t-s,W_0+B_s)\Bigr)\,ds.
\end{align}
\end{lemma}
\begin{proof} For $\lambda_i\ge0$, $x_i\in\R$ (i=1,2) and $\delta>0$ let $\vec{\lambda}=(\lambda_1,\lambda_2)$, $\vec{x}=(x_1,x_2)$ and $V_t(x)=V_t^{\delta,\vec{\lambda},\vec{x}}(x)$ be the unique smooth solution of
\begin{equation}\label{pde2sing}
\frac{\partial V}{\partial t}=\frac{\Delta V}{2}-\frac{V^2}{2},\quad V_0(\cdot)=\lambda_1p_\delta(\cdot-x_1)+\lambda_2p_\delta(\cdot-x_2),
\end{equation}
so that 
\begin{equation}\label{LLE2sing} E^X_{X_0}(e^{-\lambda_1P'_\delta X_t(x_1)-\lambda_2P'_\delta X_t(x_2)})=\exp(-X_0(V_t))
\end{equation}
 (recall \eqref{semipde} and \eqref{LLEc}).  Let  $U^{(j)}_t(x)\equiv U^{(j),\delta,\vec{\lambda},\vec{x}}_t(x)$ denote the unique solution of 
 \begin{equation}\label{Ujpde}
 \frac{\partial U_t^{(j)}}{\partial t}=\frac{\Delta U_t^{(j)}}{2}-V^{\delta,\vec{\lambda},\vec{x}}_tU^{(j)}_t,\quad U_0^{(j)}(\cdot)=p_\delta(\cdot-x_j),
 \end{equation}
 so that by Feynmann-Kac (see p. 268 of \cite{ks91}),
 \begin{equation}\label{Ujdef}
 U^{(j)}_t(x)=E_x^B\Bigl(p_\delta(B_t-x_j)\exp\Bigl(-\int_0^tV^{\delta,\vec{\lambda},\vec{x}}_{t-s}(B_s)\,ds\Bigr)\Bigr).
 \end{equation}
 Next define
 \begin{equation}\label{tildeVdefn}\tilde V^{\delta,\vec{\lambda},\vec{x}}_t(x)=V_t^{\delta,0,\lambda_2,\vec{x}}(x)+\int_0^{\lambda_1} U^{(1),\delta,\lambda,\lambda_2,\vec{x}}_t(x)\,d\lambda.
 \end{equation}
In the above it is easy to justify differentiation with respect to $t$ and $x$ through the integral and so by \eqref{Ujpde} we have
\begin{equation}\label{pdetildeV1}
\frac{\partial \tilde V}{\partial t}-\frac{\Delta\tilde V}{2}=-\frac{(V^{\delta,0,\lambda_2,\vec{x}}_t)^2}{2}-\int_0^{\lambda_1}V^{\delta,\lambda,\lambda_2,\vec{x}}_t U^{(1),\delta,\lambda,\lambda_2,\vec{x}}_t\,d\lambda,
\end{equation}
and by integration by parts,
\begin{equation}\label{pdetildeV2}
\frac{(\tilde V^{\delta,\vec{\lambda},\vec{x}}_t)^2}{2}=\frac{(V^{\delta,0,\lambda_2,\vec{x}}_t)^2}{2}+\int_0^{\lambda_1}V^{\delta,\lambda,\lambda_2,\vec{x}}_t U^{(1),\delta,\lambda,\lambda_2,\vec{x}}_t\,d\lambda.
\end{equation}
A quick check of the initial condition at $t=0$ and a comparison of \eqref{pdetildeV1} and \eqref{pdetildeV2} show that $\tilde V$ satisfies \eqref{pde2sing} and so $\tilde V=V$. 
Continuity of $V$ in $\vec{\lambda}$ is clear from \eqref{LLE2sing} and hence continuity of 
$U^{(j)}$ in $\vec{\lambda}$ follows from \eqref{Ujdef}. This allows us to differentiate \eqref{tildeVdefn} with respect to $\lambda_1$ and conclude
\begin{equation}\label{Vderiv2}
\frac{\partial V^{\delta,\vec{\lambda},\vec{x}}}{\partial\lambda_1}=U^{(1),\delta,\vec{\lambda},\vec{x}},\hbox{ and symmetrically, }\frac{\partial V^{\delta,\vec{\lambda},\vec{x}}}{\partial\lambda_2}=U^{(2),\delta,\vec{\lambda},\vec{x}}.
\end{equation}

We can differentiate the left-hand side of \eqref{LLE2sing} (set $X_0=\delta_x$) with respect to $\lambda_1$ and then $\lambda_2$ under the integral. This shows that $U^{\delta,\vec{\lambda},\vec{x}}_t(x)=\frac{\partial^2}{\partial\lambda_2\partial\lambda_1}(V^{\delta,\vec{\lambda},\vec{x}})$ exists and is continuous in $\vec{\lambda}$, and
the differentiation yields (use \eqref{Vderiv2})
\begin{align}\label{Texpansion}
E^X_{X_0}\Bigl(&\prod_{i=1}^2(P'_\delta X_t(x_i)e^{-\lambda_iP'_\delta X_t(x_i)})\Bigr)\cr
&=
\exp(-X_0(V^{\delta,\vec{\lambda},\vec{x}}_t))[X_0(U^{(1),\delta,\vec{\lambda},\vec{x}})X_0(U^{(2),\delta,\vec{\lambda},\vec{x}})-X_0(U^{\delta,\vec{\lambda},\vec{x}})]\cr
&\equiv T_1(\delta,\vec{\lambda},\vec{x})-T_2(\delta,\vec{\lambda},\vec{x}).
\end{align}

It is clear from \eqref{LLEc} that if $V^{\delta,\lambda_i,x_i}$ is the solution to \eqref{pde2sing} with initial condition $V_0^{\delta,\lambda_i,x_i}(\cdot)=\lambda_ip_\delta(\cdot-x_i)$, then
\begin{equation}\label{order}
V^{\delta,\vec{\lambda},\vec{x}}\ge V^{\delta,\lambda_i,x_i}\quad\hbox{ for }i=1,2.
\end{equation}
Using the above and \eqref{Ujdef}, we see that
\begin{align*}&\int\int\phi_1(x_1)\phi_2(x_2)T_1(\delta,\vec{\lambda},x_1,x_2)\,dx_1dx_2\cr
&\le\int\int\prod_{i=1}^2\phi_i(x_i)E^B_0\Bigl(p_\delta(W_0+B_t-x_i)\cr
&\phantom{\le\int\int\prod_{i=1}^2\phi_i(x_i)E^B_0\Bigl(p_\delta(}\times\exp\Bigl(-\int_0^t V^{\delta,\lambda_i,x_i}_{t-s}(W_0+B_s)\,ds\Bigr)\Bigr)\,dx_1dx_2\cr
\end{align*}
The above equals
\begin{align}\label{T1expression}
\int\int&\prod_{i=1}^2\phi_i(x_i)E^B_0\Bigl(p_\delta(W_0+B_t-x_i)\cr
&\phantom{\le\int\int\prod_{i=1}^2\phi_i(x_i)E^B_0\Bigl(p_\delta(}\times\exp\Bigl(-\int_0^t V^{\delta,\lambda_i,0}_{t-s}(W_0+B_s-x_i)\,ds\Bigr)\Bigr)\,dx_1dx_2\cr
&=\prod_{i=1}^2E_0^B\Bigl(\phi_i(W_0+B_{t+\delta})\exp\Bigl(-\int_0^t V^{\delta,\lambda_i,0}_{t-s}(B_s-B_{t+\delta})ds\Bigr)\Bigr).
\end{align}
An elementary argument using \eqref{LLEc} shows that 
\begin{equation}\label{deltaVlimit}
\lim_{\delta\downarrow 0}\exp(-V^{\delta,\lambda_i,0}_{t-s}(B_s-B_{t-\delta}))=\exp(-V^{\lambda_i}_{t-s}(B_s-B_t)).
\end{equation}
The elementary bound $V(\phi)(t,x)\le P'_t(\phi)(x)$ for $\phi$ non-negative, bounded and measurable shows that 
\begin{equation}\label{Vdeltabound}
V^{\delta,\lambda_i,0}_{t-s}(B_s-B_{t+\delta})\le P'_{t-s}(\lambda_ip_\delta)(B_s-B_{t+\delta})\le \lambda_i(t-s)^{-1/2}.
\end{equation}
The above two results allows us to apply Dominated Convergence in \eqref{T1expression} and conclude that
\begin{align*}&\limsup_{\delta\downarrow0}\int\int \phi_1(x_1)\phi_2(x_2)T_1(\delta,\vec{\lambda},x_1,x_2)\,dx_1dx_2\cr
&\le\prod_{i=1}^2E^B_0\Bigl(\phi_i(W_0+B_t)\exp\Bigl(-\int_0^t V^{\lambda_i}(t-s,B_t-B_s)ds\Bigl)\Bigl)\cr
&=\prod_{i=1}^2E^B_0\Bigl(\phi_i(W_0+B_t)\exp\Bigl(-\int_0^t V^{\lambda_i}(s,B_s)ds\Bigl)\Bigl)\hbox{ (time reversal)}\cr
&=\prod_{i=1}^2E^B_0\Bigl(\phi_i(W_0+\lambda_i^{-1}B_{\lambda_i^2t})\exp\Bigl(-\int_0^{\lambda_i^2t} V^{1}(u,B_u)du\Bigl)\Bigl)\cr
&\phantom{=\prod_{i=1}^2E^B_0\Bigl(\phi_i(W_0+\lambda^{-1}B_{\lambda^2t})}\hbox{ (by the scaling \eqref{pdescale})}\cr
&\le\prod_{i=1}^2 E_m^B\Bigl(\phi_i(W_0+\lambda_i^{-1}B_{\lambda_i^2t-1})\cr
&\phantom{=\prod_{i=1}^2E^B_0\Bigl(\phi_i}\times \exp\Bigl(-\int_0^{\lambda_i^2t-1}(1+u)V^1(1+u,B_u)\Bigr)(1+u)^{-1}\,du\Bigr),
\end{align*}
where we have used the Markov property at $t=1$ in the last line. Now proceed as in the proof of Proposition~\ref{LTlimit}, using Lemma~\ref{prel}(b), to conclude that
\begin{align}\label{T1limit}
\limsup_{\delta\downarrow0}\int\int \phi_1(x_1)\phi_2(x_2)T_1(\delta,\vec{\lambda},x_1,x_2)\,dx_1dx_2\cr
\le C\prod_{i=1}^2E^Y_m(\phi_i(W_0+\sqrt tY_{T^i})1(\rho>T^i)).
\end{align}

Consider next the contribution from $T_2$ in \eqref{Texpansion}. By \eqref{Vderiv2} and \eqref{Ujdef} we have
\[{\partial\over \partial\lambda_2}V^{\delta,\vec{\lambda},\vec{x}}_{t-s}(x)\le E^B_x(p_\delta(B_{t-s}-x_2))\le (t-s)^{-1/2}.\]
As the above bound is Lebesgue integrable, the Dominated Convergence Theorem and \eqref{Vderiv2} give
\[\frac{\partial}{\partial\lambda_2}\int_0^tV^{\delta,\vec{\lambda},\vec{x}}_{t-s}(B_s)ds=\int_0^tU^{(2),\delta,\vec{\lambda},\vec{x}}_{t-s}(B_s)ds.\]
This in turn allows us to differentiate \eqref{Ujdef} (with $j=1$) with respect to $\lambda_2$ and conclude
\begin{align*} U^{\delta,\vec{\lambda},\vec{x}}_t(x)=-E^B_x\Bigl(p_\delta(B_t-x_1&)\exp\Bigl(-\int_0^t V^{\delta,\vec{\lambda},\vec{x}}_{t-r}(B_r)dr\Bigr)\cr
&\times\int_0^tU^{(2),\delta,\vec{\lambda},\vec{x}}_{t-s}(B_s)ds\Bigr).
\end{align*}
Therefore $-T_2(\delta,\vec{\lambda},\vec{x})\ge 0$ and 
\begin{align*}
-&\int\int T_2(\delta,\vec{\lambda},\vec{x})\phi_1(x_1)\phi_2(x_2)dx_1dx_2\cr
&\le\int\int E_0^B\Bigl(p_\delta(W_0+B_t-x_1)\exp\Bigl(-\int_0^t V^{\delta,\vec{\lambda},\vec{x}}_{t-r}(W_0+B_r)dr\Bigr)\phi_1(x_1)\cr
&\phantom{\int\int E_0^B\Bigl(p_\delta}\times\int_0^tU^{(2),\delta,\vec{\lambda},\vec{x}}_{t-s}(W_0+B_s)ds\Bigr)\phi_2(x_2)dx_1dx_2\Bigr)\cr
&\le \int_0^tE_0^B\Bigl(\int p_\delta(W_0+B_t-x_1)\exp\Bigl(-\int_0^t V^{\delta,\lambda_1,x_1}_{t-r}(W_0+B_r)dr\Bigr)\phi_1(x_1)\cr
&\phantom{}\times\int E^{\hat B}_{B(s)+W_0}\Bigl(p_\delta(\hat B_{t-s}-x_2)\exp\Bigl(-\int_0^{t-s}V^{\delta,\lambda_2,x_2}_{t-s-u}(\hat B_u)du\Bigr)\Bigr)\phi_2(x_2)dx_2dx_1\Bigr)ds.
\end{align*}
In the last line $\hat B$ is a Brownian motion and we have used \eqref{order} and \eqref{Ujdef}.  The above equals
\begin{align*}
&\int_0^tE_0^B\Bigl(\phi_1(W_0+B_{t+\delta})\exp\Bigl(-\int_0^t V^{\delta,\lambda_1,0}_{t-r}(B_r-B_{t+\delta})dr\Bigr)\cr
&\phantom{\int_0^tE_0^B\Bigl(}\times E^{\hat B}_{B_s+W_0}\Bigl(\phi_2(\hat B_{t-s+\delta})\exp\Bigl(-\int_0^{t-s} V^{\delta,\lambda_2,0}_{t-s-u}(\hat B_u-\hat B_{t-s})du\Bigr)\Bigr)\Bigr)ds\cr
&\rightarrow\int_0^tE_0^B\Bigl(\phi_1(W_0+B_t)\exp\Bigl(-\int_0^t V^{\lambda_1}_{t-r}(B_t-B_r)dr\Bigr)\cr
&\phantom{\int_0^tE_0^B\Bigl(}\times E^{\hat B}_{B_s+W_0}\Bigl(\phi_2(\hat B_{t-s})\exp\Bigl(-\int_0^{t-s} V^{\lambda_2}_{t-s-u}(\hat B_{t-s}-\hat B_u)du\Bigr)\Bigr)\Bigr)ds,
\end{align*}
as $\delta\downarrow0$.  Here we used Dominated Convergence and \eqref{Vdeltabound} as in \eqref{deltaVlimit}.  If $B'_u=\hat B_{t-s}-\hat B_{t-s-u}$ (a Brownian motion starting at $0$), then the above shows that 
\begin{align}\label{T2limit}
&\limsup_{\delta\downarrow 0}\ \left[-\int\int T_2(\delta,\vec{\lambda},\vec{x})\phi_1(x_1)\phi_2(x_2)dx_1dx_2\right]\cr
&\le \int_0^tE_0^B\Bigl(\phi_1(W_0+B_t)\exp\Bigl(-\int_0^tV^{\lambda_1}_{t-r}(B_t-B_r)dr\Bigr)\cr
&\phantom{\le \int_0^tE_0^B\Bigl(}\times E^{B'}_{B_s+W_0}(\phi_2(W_0+B_s+B'_{t-s})\exp\Bigl(-\int_0^{t-s}V^{\lambda_2}_r(B'_r)dr\Bigr)\Bigr)\Bigr)ds\cr
&= \int_0^tE_0^B\Bigl(\phi_1(W_0+B_t)\exp\Bigl(-\int_0^tV^{\lambda_1}(t-r,B_t-B_r)dr\Bigr)\cr
&\phantom{\int_0^tE_0^B\Bigl(\phi_1(W_0+}\times\psi^{\lambda_2}(t-s,W_0+B_s)\Bigr)ds.
\end{align}
The last line is an easy consequence of the scaling relation \eqref{pdescale}. Now use \eqref{Texpansion} and Fatou's lemma to see that
\begin{align*}E^X_{X_0}&\Bigl(\int\int\phi_1(x_1)\phi_2(x_2)X_t(x_1)X_t(x_2)e^{-\lambda_1 X_t(x_1)-\lambda_2 X_t(x_2)}dx_1dx_2\Bigr)\cr
&\le\liminf_{\delta\downarrow 0}\Bigl[\int\int\prod_{i=1}^2\phi_i(x_i) T_1(\delta,\vec{\lambda},\vec{x})dx_1dx_2
\cr
&\phantom{\le\liminf_{\delta\downarrow 0}\int\int}+\int\int\prod_{i=1}^2\phi_i(x_i)(-T_2(\delta,\vec{\lambda},\vec{x}))dx_1dx_2\Bigr].\cr
\end{align*}
Finally apply \eqref{T1limit} and \eqref{T2limit} to bound the above by the required expression.
 \end{proof}
\noindent{\bf Proof of Proposition~\ref{secmombnd}}\label{5.1}. Let $T_1$ and $T_2$ denote the first and second terms, respectively on the right-hand side of \eqref{secmom1bnd}, where $t$ and $\lambda=\lambda_1=\lambda_2$ are fixed as in Proposition~\ref{secmombnd}, and let $T=\log(\lambda^2 t)$.  Consider first the much easier $T_1$. Recall from \eqref{rhotail}-\eqref{psi0bnd} we have for any $\delta>0$,
\begin{equation}\label{extbnd} P^Y_x(\rho>T)\le c_\delta e^{\delta x^2}e^{-\lambda_0 T}.
\end{equation}
Therefore 
\begin{align*} \lambda^{2\lambda_0}&E^Y_m(\phi_i(W_0+\sqrt t Y_T)1(\rho>T))\cr
&= \lambda^{2\lambda_0}\int\int\int\phi_i(w_0+\sqrt t y)q_T(y_0,y)dm(y_0)dm(y)dX_0(w_0)\cr
&\le  \int\int\phi_i(w_0+\sqrt ty)\lambda^{2\lambda_0}P^Y_y(\rho>T)dm(y)dX_0(w_0)\cr
&\le \int\int\phi_i(w_0+\sqrt t y)t^{-\lambda_0}c_\delta e^{\delta y^2}dm(y) dX_0(w_0),
\end{align*}
the last by \eqref{extbnd}.  Now let $\sigma^2=(1-2\delta)^{-1}$. A simple substitution now shows the above is at most
\[ct^{-\lambda_0}\int\phi_i(x_i) X_0p_{t\sigma^2}(x_i)dx_i.\]
This in turn implies
\begin{equation}\label{T2bound}
\lambda^{4\lambda_0}T_1\le c_{\sigma^2}t^{-2\lambda_0}\int \int \prod_{i=1}^2 [\phi_i(x_i) X_0p_{t \sigma^2}(x_i)]\,dx_1dx_2,
\end{equation}
where here $\sigma^2$ is any number greater than 1. Below we will choose a convenient value of $\sigma^2$ when doing the $T_2$ bound.  

Turning now to $T_2$, we can write $T_2=\int_0^tT_2(s)ds$, where $T_2(s)$ is the integrand on the right-hand side of \eqref{secmom1bnd}. We may replace the Brownian motion $B_t$ with $\lambda^{-1}B_{t\lambda^2}$ (the new $B$ is still a Brownian motion starting at $0$) and use the scaling relation \eqref{pdescale} to conclude after a short and familiar argument that
\begin{align}\label{T1bnd3}
T_2(s)= E_0^B&\Bigl(\phi_1(W_0+\lambda^{-1}B_{t\lambda^2})\psi^\lambda(t-s, W_0+\lambda^{-1}(B_{\lambda^2t}-B_{\lambda^2(t-s)}))\cr
&\times\exp\Bigl(-\int_0^{\lambda^2 t}V^1(r,B_r)\,dr\Bigr)\Bigr).
\end{align}
{\bf Case 1.} Assume $\lambda^2(t-s)>1$.  \hfil\break
Apply the Markov property of $B$ at time $1$ to the right-hand side of \eqref{T1bnd3} and conclude that
\begin{align}\label{T1bnd4}
T_2(s)\le E^B_m\Bigl(\phi_1&(W_0+\lambda^{-1}B_{t\lambda^2-1})\psi^\lambda(t-s,\lambda^{-1}(B_{t\lambda^2-1}-B_{\lambda^2(t-s)-1}))\cr
&\times\exp\Bigl(-\int_0^{\lambda^2t-1}V^1(r+1,B_r)\,dr\Bigr)\Bigr),
\end{align}
where $B$ remains independent of $W_0$.  As before, $Y(u)=B(e^u-1)e^{-u/2}$ is a stationary Ornstein-Uhlenbeck process. If $U=\log(\lambda^2(t-s))$, then arguing as in the proof of Lemma~\ref{LTequation}, we may re-express \eqref{T1bnd4} as
\begin{align}\label{T1bnd4.5}
T_2(s)&\le E^Y_m\Bigl(\phi_1(W_0+\sqrt t Y_T)\psi^\lambda(t-s,W_0+\sqrt t Y_T-\sqrt{t-s}Y_U)\cr
&\qquad\times\exp\Bigl(-\int_0^T H_{e^u}(Y_u)\,du\Bigr)\Bigr)\cr
&\le e^{C_Z}E^Y_m\Bigl(\phi_1(W_0+\sqrt t Y_T)\psi^\lambda(t-s,W_0+\sqrt t Y_T-\sqrt{t-s}Y_U)\cr&\qquad\times\exp\Bigl(-\int_0^T F(Y_u)\,du\Bigr)\Bigr),
\end{align}
where we used Lemma~\ref{prel} in the last inequality.
The above equals
\begin{align}
&e^{C_Z}\int \phi_1(w_0+\sqrt ty_2)\psi^\lambda(t-s,w_0+\sqrt t y_2-\sqrt{t-s}\,y_1)q_U(y_0,y_1)\cr
&\qquad\qquad\times q_{T-U}(y_1,y_2)dm(y_0)dm(y_1)dm(y_2)dX_0(w_0)\cr
&\le C_{\delta_1} \int \phi_1(w_0+\sqrt ty_2)\psi^\lambda(t-s,w_0+\sqrt t y_2-\sqrt{t-s}\,y_1)\cr
&\qquad\qquad\times e^{\delta_1 y^2_1}(t-s)^{-\lambda_0}\lambda^{-2\lambda_0}q_{T-U}(y_1,y_2)dm(y_1)dm(y_2)dX_0(w_0),\label{T1bnd5}
\end{align}
 where in the last line, $1/2>\delta_1>0$, and we used \eqref{extbnd} and the symmetry of $q_U$ in integrating out $y_0$. 

At this point we take a break from the long proof and obtain a bound on $\psi^\lambda$. 
\begin{lemma} \label{psilambnd}For any $\sigma_0^2>1$ there is a $C_{\ref{psilambnd}}(\sigma_0^2)$ such that for all $x$ and all $0\le s <t$, 
\begin{equation*} \psi^\lambda(t-s,x)\le C_{\ref{psilambnd}} \lambda^{-2\lambda_0}(t-s)^{-\lambda_0}P'_{\sigma_0^2(t-s)}\phi_2(x).
\end{equation*}
\end{lemma}
\begin{proof} If $\lambda^2(t-s)\le 1$, we can drop the negative exponential in the definition of $\psi^\lambda$ and note that $\lambda^{-2\lambda_0}(t-s)^{-\lambda_0}\ge 1$
to conclude that 
\[\psi^\lambda(t-s,x)\le \lambda^{-2\lambda_0}(t-s)^{-\lambda_0}P'_{t-s}\phi_2(x),\]
from which the required bound follows easily. 

Assume now that $\lambda^2(t-s)>1$.  If $Y$ and $U$ are as above, then the same reasoning leading to \eqref{T1bnd4.5} and \eqref{T1bnd5} above (compare the definition of $\psi^\lambda$ with the right-hand side of \eqref{T1bnd3} without the $\psi^\lambda$) leads to (for any $0<\delta<1/2$)
\begin{align}
\psi^\lambda(t-s,x)&\le CE^Y_m\Bigl(\phi_2(\sqrt{t-s}Y_U+x)\exp\Bigl(-\int_0^UF(Y_u)\,du\Bigr)\Bigr)\cr
&\le C_\delta(t-s)^{-\lambda_0}\lambda^{-2\lambda_0}\int\phi_2(x+\sqrt{t-s}\,y)e^{\delta y^2}dm(y).
\end{align}
An elementary argument now gives the required bound with $\sigma_0^2=(1-2\delta)^{-1}$. 
\end{proof}
Returning to the proof of Proposition~\ref{secmombnd} in Case 1, we use the above lemma 
in \eqref{T1bnd5} and then the substitution $z_1=w_0+\sqrt t y_2$, to obtain
\begin{align}\label{T1bnd6}
T_2(s)&\le C(\sigma_0^2,\delta_1)\lambda^{-4\lambda_0}(t-s)^{-2\lambda_0}\cr
&\quad\times\int\int\int\phi_1(w_0+\sqrt ty_2)
P'_{\sigma^2_0(t-s)}\phi_2(w_0+\sqrt ty_2-\sqrt{t-s}\,y_1)e^{\delta_1 y_1^2}\cr
&\phantom{\le C(\sigma_0^2,\delta_1)\lambda^{-4\lambda_0}(t-s)^{-2\lambda_0}}\quad\times q_{T-U}(y_1,y_2)dm(y_1)dm(y_2)dX_0(w_0)\cr
&\le C(\sigma_0^2,\delta_1)\lambda^{-4\lambda_0}(t-s)^{-2\lambda_0}\int\int\phi_1(z_1)\phi_2(z_2)\cr
&\times\Bigl[\int\int p_{\sigma^2_0(t-s)}(z_2-z_1+\sqrt{t-s}\,y_1)q_{\log(t/t-s)}(y_1,(z_1-w_0)t^{-1/2})\cr
&\qquad\times e^{\delta_1 y_1^2}p_t(z_1-w_0)dm(y_1)dX_0(w_0)\Bigr]dz_1dz_2.
\end{align}
{\bf Case 1a} Assume also $t/2\le s$.\hfil\break
Let $\delta=\frac{\sqrt 2-1}{2}$ for which $e^{-s^*(\delta)}=1/2$ (recall that $s^*(\delta)$ is as in \eqref{s*def}).  Therefore $\log(t/t-s)\ge \log 2=s^*(\delta)$, and so by \eqref{qbound1},
\[q_{\log(t/t-s)}(y_1,(z_1-w_0)t^{-1/2})\le c e^{-\lambda_0\log(t/t-s)}\exp(\delta(y_1^2+(z_1-w_0)^2t^{-1})).\]
If $\delta_1=\delta$ and $\sigma_0^2=(1-4\delta)^{-1}=(3-2\sqrt 2)^{-1}\le 6$, this implies that the expression in square brackets in \eqref{T1bnd6} is at most
\begin{align*}
c&t^{-\lambda_0}(t-s)^{\lambda_0}\int\int p_{\sigma_0^2(t-s)}(z_2-z_1+\sqrt{t-s}\,y_1)e^{2\delta y_1^2}dm(y_1)\cr
&\phantom{t^{-\lambda_0}(t-s)^{\lambda_0}\int\int p}\times\exp(\delta(z_1-w_0)^2/t)p_t(z_1-w_0)dX_0(w_0)\cr
&\le c t^{-\lambda_0}(t-s)^{\lambda_0}\int\int p_{\sigma_0^2(t-s)}(z_2-z_1+w)\exp(-(1-4\delta)w^2/(2(t-s)))\cr
&\qquad\times(2\pi(t-s))^{-1/2}dw\exp(-(1-2\delta)(z_1-w_0)^2/2t)(2\pi t)^{-1/2}dX_0(w_0)\cr
&\le ct^{-\lambda_0}(t-s)^{\lambda_0}\int p_{\sigma_0^2(t-s)}(z_2-z_1+w)p_{\sigma_0^2(t-s)}(w)dw X_0p_{\sigma_0^2 t}(z_1)\cr
&=ct^{-\lambda_0}(t-s)^{\lambda_0}p_{2\sigma_0^2(t-s)}(z_2-z_1)X_0p_{\sigma_0^2 t}(z_1).
\end{align*}
Use the above in \eqref{T1bnd6} to conclude that in Case 1a,
\begin{align}\label{T1bnd6.5}
T_2(s)\le C&\lambda^{-4\lambda_0}(t-s)^{-\lambda_0}t^{-\lambda_0}\cr
&\times\int\phi_1(z_1)\phi_2(z_2)p_{2\sigma_0^2(t-s)}(z_2-z_1)X_0p_{\sigma_0^2 t}(z_1)dz_1dz_2.
\end{align}

\noindent{\bf Case 1b.} Assume $0\le s<t/2\ (\le t-\lambda^{-2})$.\hfil\break
The last inequality is immediate by our hypothesis that $\lambda^2 t\ge 9>2$. 
Return to \eqref{T1bnd6} with the choices of $\delta_1$ and $\sigma_0^2$ made in the previous case and let  $R=\log(t/t-s)$.  Bounding the transition density (with respect to Lebesgue measure) of the killed Ornstein-Uhlenbeck process starting at $y_1$ by the unkilled process and noting the latter has a normal density with mean $y_1 e^{-R/2}$ and variance $1-e^{-R}$, we get
\[ q_R(y_1,z)e^{-z^2/2}(2\pi)^{-1/2}\le\exp\Bigl(\frac{-(z-y_1e^{-R/2})^2}{2(1-e^{-R})}\Bigr)(2\pi)^{-1/2}(1-e^{-R})^{-1/2}.\]
Setting $z=\frac{z_1-w_0}{\sqrt t}$ and simplifying, this becomes 
\begin{equation*}
q_R(y_1,(z_1-w_0)/\sqrt t)p_t(z_1-w_0)\le p_s(z_1-w_0-\sqrt{t-s}y_1).
\end{equation*}
Now use this in \eqref{T1bnd6} to conclude
\begin{align}\label{T1bnd7}
T_2(s)\le &C\lambda^{-4\lambda_0}(t-s)^{-2\lambda_0}\int\int\phi_1(z_1)\phi_2(z_2)\cr
&\times\int\Bigl[\int p_{\sigma_0^2(t-s)}(z_2-z_1+\sqrt{t-s}\,y_1)e^{\delta_1y_1^2}\cr
&\qquad\qquad\times p_s(z_1-w_0-\sqrt{t-s}\,y_1)dm(y_1)\Bigr]dX_0(w_0)dz_1dz_2.
\end{align}
The fact that $\sigma_0^2=(1-4\delta_1)^{-1}>(1-2\delta_1)^{-1}>1$ and a simple substitution shows that the term in square brackets is at most
\begin{equation}f_{s,t}(z_1,z_2,w_0)=\int p_{\sigma_0^2(t-s)}(z_2-z_1+w)p_{\sigma_0^2(t-s)}(w)p_{\sigma_0^2 s}(z_1-w_0-w)\,dw.
\end{equation}

We claim that 
\begin{equation}\label{fbound0}
f_{s,t}(z_1,z_2,w_0)\le cp_{8\sigma_0^2 (t-s)}(z_2-z_1)p_{4\sigma_0^2 t}(z_1-w_0).
\end{equation}
By scaling it suffices to obtain the above for $\sigma_0=1$.  Set $a=z_1-z_2$ and $b=z_1-w_0$, so that 
\[f_{s,t}(z_1,s_2,w_0)=\int p_{t-s}(w-a)p_{t-s}(w)p_s(b-w)dw.\]
A simple calculation (complete the square) shows that 
\[f_{s,t}(z_1,s_2,w_0)=(2\pi)^{-1}(t^2-s^2)^{-1/2}\exp\Bigl(-\frac{a^2t+2b^2(t-s)-2ab(t-s)}{2(t^2-s^2)}\Bigr).\]
Now use $ab\le \alpha a^2+(4\alpha)^{-1}b^2$ with $\alpha=5/16$ to see that
\begin{align*} f_{s,t}(z_1,s_2,w_0)&\le(2\pi)^{-1}(t^2-s^2)^{-1/2} \exp\Bigl(-\frac{a^2(1-2\alpha)}{2(t+s)}\Bigr)\exp\Bigl(-\frac{b^2(2-(2\alpha)^{-1})}{2(t+s)}\Bigr)\cr
&\le  (2\pi)^{-1}(t^2-s^2)^{-1/2}\exp\Bigl(-\frac{a^2(1-2\alpha)}{6(t-s)}\Bigr)\exp\Bigl(-\frac{b^2(2-(2\alpha)^{-1})}{3t}\Bigr)\cr
&\le (2\pi)^{-1}(t-s)^{-1/2} \exp\Bigl(-\frac{a^2}{16(t-s)}\Bigr)t^{-1/2}\exp\Bigl(-\frac{b^2}{8t}\Bigr),\cr
\end{align*}
where we use $s< t/2$ in the next to last line and the value of $\alpha$ in the last line. This completes the proof of \eqref{fbound0}.

Now insert \eqref{fbound0} into \eqref{T1bnd7}, noting that $(t-s)^{-\lambda_0}\le ct^{-\lambda_0}$,  to conclude that in Case 1b,
\begin{align}\label{T1bnd8}T_2(s)&\le C\lambda^{-4\lambda_0}(t-s)^{-\lambda_0}t^{-\lambda_0}\cr
&\qquad\times\int\int\phi_1(z_1)\phi_2(z_2)p_{8\sigma_0^2(t-s)}(z_2-z_1)X_0p_{4\sigma_0^2t}(z_1)\,dz_1dz_2.
\end{align}

\noindent{\bf Case 2.} Assume $\lambda^2(t-s)\le 1$. 
\hfil\break
Use \eqref{T1bnd3}, the Markov property at time $1$ and then argue as in \eqref{T1bnd5} to see that
\begin{align*}
T_2(s)&\le E^B_0\Bigl(E^B_{B_1(\omega)}\Bigl(\phi_1(W_0+\lambda^{-1}B_{\lambda^2t-1})\psi^\lambda(t-s,W_0+\lambda^{-1}(B_{\lambda^2t-1}-B_{\lambda^2(t-s)}(\omega)))\cr
&\qquad\times\exp\Bigl(-\int_0^{\lambda^2t-1}V^1(r+1,B_r)\,dr\Bigr)\Bigr)\Bigr)\cr
&\le C\int\int p_{\lambda^2(t-s)}(x_0)p_{1-\lambda^2(t-s)}(x_1-x_0)\cr
&\qquad\times E^Y_{x_1}\Bigl(\phi_1(W_0+\sqrt t Y_T)\psi^\lambda(t-s,W_0+\sqrt tY_T-x_0\lambda^{-1})\cr
&\qquad\qquad\qquad\times\exp\Bigl(-\int_0^TF(Y_u)du\Bigr)\Bigr) dx_1dx_0.
\end{align*}
By definition, $\psi^\lambda(t-s,x)\le P'_{t-s}\phi_2(x)$, and so arguing as in the derivation of \eqref{T1bnd6} we get
\begin{align}\label{T1bnd9}
T_2(s)&\le C\int\int\phi_1(z_1)\phi_2(z_2)\Bigr[\int\int\int p_{t-s}(z_2-z_1+\frac{x_0}{\lambda})p_{\lambda^2(t-s)}(x_0)\cr
&\phantom{\le C\int\int\phi_1(z_1)\phi_2(z_2)\Bigr[\int}\times p_{1-\lambda^2(t-s)}(x_1-x_0)q_T(x_1,(z_1-w_0)t^{-1/2})\cr
&\phantom{\le C\int\int\phi_1(z_1)\phi_2(z_2)\Bigr[\int}\quad\times p_t(z_1-w_0)dx_1dx_0dX_0(w_0)\Bigr]dz_1dz_2.
\end{align}
Let $g_{s,t}(z_1,z_2)$ denote the expression in square brackets.  A simple calculation shows that our condition $\lambda^2t\ge 9$ implies $T\ge s^*(1/8)$, and so by \eqref{qbound1}, 
\[q_T(x_1,(z_1-w_0)t^{-1/2})\le C(\lambda^2 t)^{-\lambda_0}\exp\Bigl(\frac{1}{8}(x_1^2+\frac{(z_1-x_0)^2}{t})\Bigr).\]
First use this in \eqref{T1bnd9}, and then set $\sigma_1^2=4/3$ and use $1-\lambda^2(t-s)\le 1$ and an easy calculation to obtain
\begin{align*}
g_{s,t}(z_1,z_2)&\le C\lambda^{-2\lambda_0}t^{-\lambda_0}\int\int p_{t-s}(z_2-z_1+(x_0/\lambda))p_{\lambda^2(t-s)}(x_0)\cr
&\phantom{\le C\lambda^{-2\lambda_0}t^{-\lambda_0}}\times p_{1-\lambda^2(t-s)}(x_1-x_0)e^{x_1^2/8}dx_1dx_0p_{\sigma^2t}(z_1-w_0)dX_0(w_0)\cr
&\le C\lambda^{-2\lambda_0}t^{-\lambda_0}\int p_{t-s}(z_2-z_1+(x_0/\lambda))p_{\lambda^2(t-s)}(x_0)\cr
&\phantom{\le C\lambda^{-2\lambda_0}t^{-\lambda_0}\int\int }\times e^{x_0^2/4}dx_0X_0p_{\sigma_1^2t}(z_1).
\end{align*}
If $\sigma^2_2=\frac{\lambda^2(t-s)}{1-\lambda^2(t-s)/2}$, then 
\begin{equation}\label{sigma2}
\lambda^2(t-s)\le \sigma_2^2\le 2\lambda^2(t-s)
\end{equation}
and so 
\[e^{x_0^2/4}p_{\lambda^2(t-s)}(x_0)=p_{\sigma_2^2}(x_0)\sqrt{\frac{\sigma_2^2}{\lambda^2(t-s)}}\le \sqrt{2}p_{\sigma_2^2}(x_0),\]
which in turn implies
\begin{align*}
g_{s,t}(z_1,z_2)&\le C\lambda^{-2\lambda_0}t^{-\lambda_0}\int p_{t-s}(z_2-z_1+(x_0/\lambda))p_{\sigma_2^2}(x_0)dx_0X_0p_{\sigma_1^2}(z_1)\cr
&=C\lambda^{-2\lambda_0}t^{-\lambda_0}p_{(\sigma^2_2/\lambda)+(t-s)}(z_2-z_1)X_0p_{\sigma_1^2}(z_1)\cr
&\le C\lambda^{-2\lambda_0}t^{-\lambda_0}p_{3(t-s)}(z_2-z_1)
X_0p_{\sigma_1^2}(z_1),
\end{align*}
the last by \eqref{sigma2}.  Use this in \eqref{T1bnd9} to see that in Case 2,
\begin{align}\label{T1bnd10}T_2(s)&\le C\lambda^{-2\lambda_0}t^{-\lambda_0}\int\int\phi_1(z_1)\phi_2(z_2)p_{3(t-s)}(z_2-z_1)X_0p_{\sigma_1^2 t}(z_1)dz_1dz_2\cr
&\le C\lambda^{-2\lambda_0}t^{-\lambda_0}\int\int\phi_1(z_1)\phi_2(z_2)p_{2\sigma^2_0(t-s)}(z_2-z_1)X_0p_{\sigma_0^2 t}(z_1)dz_1dz_2,
\end{align}
the last using $3<2\sigma_0^2$ and $\sigma_1^2<\sigma_0^2$.  

Now combine \eqref{T1bnd6.5}, \eqref{T1bnd8}, and \eqref{T1bnd10} to see that \eqref{T1bnd8} holds for all $0< s\le t$ and conclude that
\begin{align*}T_2\le\int&\int\phi_1(z_1)\phi_2(z_2)\cr
&\times\Bigl[C\lambda^{-4\lambda_0}t^{-\lambda_0}\int_0^t(t-s)^{-\lambda_0}p_{8\sigma_0^2(t-s)}(z_1-z_2)X_0p_{4\sigma_0^2 t}(z_1)\,ds\Bigr]dz_1dz_2.
\end{align*}
Combine this with \eqref{T2bound} to see that if 
\[f_{t,\lambda}(z_1,z_2)=\lambda^{4\lambda_0}E(X_t(z_1)X_t(z_2)e^{-\lambda X_t(z_1)-\lambda X_t(z_2)}),\]
then for $\lambda^2t\ge 9$, and $\sigma_0^2$ as above,
\[\int\phi_1(z_1)\phi_2(z_2)f_{t,\lambda}(z_1,z_2)dz_1dz_2\le \int\int \phi_1(z_1)\phi_2(z_2)Ch_{t,X_0}(z_1,z_2)dz_1dz_2.\]
This implies that $f_{t,\lambda}(z)\le Ch_{t,X_0}(z)$ for Lebesgue a.a. $z=(z_1,z_2)$. Note that Fatou's lemma shows that $h_{t,X_0}(\cdot)$ is lower semicontinuous while $f_{t,\lambda}(\cdot)$ is continuous by Dominated Convergence.  Therefore 
$$f_{t,\lambda}(z)\le Ch_{t,X_0}(z)\text{ for all }z,$$
proving (a).  

The first term in the definition of $h_{t,X_0}(z)$ is trivially bounded by \break$Ct^{-2\lambda_0-1}X_0(1)^2$.  A simple substitution in the integral shows that the second term is bounded by 
$$Ct^{-\lambda_0-1/2}X_0(1)|z_1-z_2|^{1-2\lambda_0}.$$
This proves (b), and so the proof of Proposition~\ref{secmombnd} is complete
\qed

\section{Proof of Theorem \ref{lowdensint}} \label{5.8}
\setcounter{equation}{0}

Let $\tilde V^{b,\vep}=V(b1_{[0,\vep]})$ and let $\tilde V^{\infty,\vep}=\lim_{b\to\infty} \tilde V^{b,\vep}$, where the pointwise finite limit exists by monotonicity (from \eqref{LLEc}) and 
the bound
\begin{equation}\label{Vbound}
 \tilde V^{b,\vep}(t,x)\le V(b)(t,x)=\frac{b}{1+(bt/2)}\le \frac{2}{t}.
 \end{equation}
 The scaling relation for $\tilde V^{b,\vep}$, is 
 \begin{equation}\label{scale2vb}
 \text{for each }r>0, \, \tilde V^{b,\vep}(s,x)=r^2\tilde V^{br^{-2},\vep r}(r^2s,rx).
 \end{equation}
 We state two lemmas which will be used to prove Theorem~\ref{lowdensint} and prove them after establishing the theorem. 
 \begin{lemma}\label{smalltbnd} Let $\mu_1$ denote the uniform law on $[0,1]$.  For any $\delta_0>0$ there is a $C_{\ref{smalltbnd}}(\delta_0)>1$ so that for all $\lambda\ge 1$,
 \begin{equation*}E^B_{\mu_1}\Bigl(\exp\Bigl(-\int_0^1 \tilde V^{\lambda,1}(u,B_u)\,du\Bigr)\Bigr)\le C_{\ref{smalltbnd}}\lambda^{-((1/2)+\lambda_0)+\delta_0}.
 \end{equation*}
 \end{lemma}
 The analogue of $H_u$ in \eqref{Hdef} is (for any $b>0$)
 \begin{equation}\label{Hdef2} H^b(u,x)=u\tilde V^{b,1}(u,\sqrt u x)=\tilde V^{bu,u^{-1/2}}(1,x),
 \end{equation}
 where we used \eqref{scale2vb} with $r=u^{-1/2}$. The latter expression suggests that $\lim_{u\to\infty} H^b(u,x)=V^\infty(1,x)=F(x)$.  Let $Z^b_T=\exp\Bigl(\int_0^T F(Y_s)-H^b(e^s,Y_s)\,ds\Bigr)$. With Lemma~\ref{prel}(b) in mind we have:
 
 \begin{lemma}\label{Zbnd2} There is a  constant $C_{\ref{Zbnd2}}$ such that for all $b,T>0$,
 \[Z^b_T\le C_{\ref{Zbnd2}}(b\wedge 1)^{-20}.\]
 \end{lemma}
 
 \noindent{\bf Proof of Theorem \ref{lowdensint}.}  Let $t,\vep$ and $M$ be as in the Theorem and set $\lambda=(\vep M)^{-1}\ge 1$.  By \eqref{LLEc} and translation invariance,
 \begin{align}
\label{LTboundb}
P^X_{X_0}&(0<X_t([x,x+\vep])/\vep\le \vep M)\cr
&\le e E^X_{X_0}\Bigl(1(X_t([x,x+\vep]>0)\exp(-\lambda X_t([x,x+\vep])/\vep))\Bigr)\cr
&=e\Bigl[\exp\Bigl(-\int \tilde V^{\lambda/\vep,\vep}(t,y-x)dX_0(y)\Bigr)-\exp\Bigl(-\int \tilde V^{\infty,\vep}(t,y-x)dX_0(y)\Bigr)\Bigr].
\end{align}
 Differentiate the last inequality (with $x=0$ and $X_0$ a point mass) with respect to
 $\lambda$, using Dominated Convergence to take the derivative through the integral,
 to conclude for $\lambda'>0$,
 \begin{align*}
 \frac{d}{d\lambda'}\Bigl(&\exp(-\tilde V^{(\lambda'/\vep),\vep}(t,x))-\exp(-\tilde V^{\infty,\vep}(t,x))\Bigr)\cr
 &=-\frac{d}{d\lambda'}E^X_{\delta_x}\Bigl(\exp(-(\lambda'/\vep)X_t([0,\vep]))1(X_t([0,\vep)>0))\Bigr)\cr
 &=-E^X_{\delta_x}\Bigl(\exp(-(\lambda'/\vep)X_t([0,\vep]))X_t([0,\vep])/\vep\Bigr).
 \end{align*}
 Integrate both sides from $\lambda$ to $\infty$ and then integrate out $x$ to conclude (by Fubini)
 \begin{align*} \int &e^{-\tilde V^{\lambda/\vep,\vep}(t,x)}-e^{-\tilde V^{\infty,\vep}(t,x)}\,dx\cr
 &= \int_\lambda^\infty\int E^X_{\delta_x}\Bigl(e^{-(\lambda'/\vep)X_t([0,\vep])}X_t([0,\vep])/\vep\Bigr)\,dx\,d\lambda'.
 \end{align*}
 The bound \eqref{Vbound} and the above show that
 \begin{align}\label{Vintbnd}
 \int&(\tilde V^{\infty,\vep}-\tilde V^{\lambda/\vep,\vep})(t,x)\,dx\cr
 &\le e^{2/t}\int e^{-\tilde V^{\lambda/\vep,\vep}(t,x)}-e^{-\tilde V^{\infty,\vep}(t,x)}\,dx\cr
 &\le e^{2/t}\int_\lambda^\infty\int E^X_{\delta_x}\Bigl(e^{-(\lambda'/\vep)X_t([0,\vep])}X_t([0,\vep])/\vep\Bigr)\,dx\,d\lambda'\cr
 &=e^{2/t}\int_\lambda^\infty\int E^X_{\delta_0}\Bigl(e^{-(\lambda'/\vep)X_t([x,x+\vep])}X_t([x,x+\vep])/\vep\Bigr)\,dx\,d\lambda'.
 \end{align}
Now use the Markov property at time $t/2$ just as in \eqref{Vincrement} and then apply the above to see that
\begin{align*}
e&^{-\tilde V^{\lambda/\vep,\vep}(t,x)}-e^{-\tilde V^{\infty,\vep}(t,x)}\cr
&\le \int p_{t/2}(y-x)(\tilde V^{\infty,\vep}(t/2,y)-\tilde V^{\lambda/\vep,\vep}(t/2,y))\,dy\cr
&\le e^{2/t}t^{-1/2}\int_\lambda^\infty\int E^X_{\delta_0}\Bigl(e^{-(\lambda'/\vep)X_t([y,y+\vep])}X_t([y,y+\vep])/\vep\Bigr)\,dy\,d\lambda',
\end{align*}
and so by \eqref{Vbound}, if $f_{\vep,t}(\lambda')=\int E^X_{\delta_0}\Bigl(e^{-(\lambda'/\vep)X_t([y,y+\vep])}X_t([y,y+\vep])/\vep\Bigr)\,dy$,
then
\begin{equation*}
\sup_x[\tilde V^{\infty,\vep}-\tilde V^{\lambda/\vep,\vep}](t,x)\le c(t)\int_\lambda^\infty f_{\vep,t}(\lambda')\,d\lambda'.
\end{equation*}
Returning to \eqref{LTboundb} we therefore see that (recall $\lambda=(\vep M)^{-1}$)
\begin{align}\label{fbound}
P^X_{X_0}(0<X_t([x,x+\vep])/\vep\le \vep M)&\le eX_0(1)\sup_y[\tilde V^{\infty,\vep}-\tilde V^{\lambda/\vep,\vep}](t,y)\cr
&\le c(t)X_0(1)\int_{(\vep M)^{-1}}^\infty f_{\vep,t}(\lambda')d\lambda'.\end{align}
Let $\mu_\vep$ denote the uniform distribution on $[0,\vep]$. Argue as in the derivation of \eqref{campbell} to see that 
\begin{align}f_{\vep,t}(\lambda')&=E^X_{\delta_0}\Bigl(\int e^{-(\lambda'/\vep)X_t([y,y+\vep])}\int 1_{[0,\vep]}(x-y)X(t,x)dxdy/\vep\Bigr)\cr
&=\vep^{-1}\int_0^\vep E^X_{\delta_0}\Bigl(\int X(t,x)\exp(-(\lambda'/\vep)X_t(x-z+[0,\vep]))dx\Bigr)\,dz\cr
&\le \vep^{-1}\int_0^\vep E_0^B\Bigl(\exp\Bigl(-\int_0^t \tilde V^{\lambda'/\vep,\vep}(t-s, B_s-B_t+z)ds\Bigr)\Bigr)\,dz\cr
&=E^B_{\mu_\vep}\Bigl(\exp\Bigl(-\int_0^t\tilde V^{\lambda'/\vep,\vep}(s,B_s)ds\Bigr)\Bigr).
\end{align}
Now apply the scaling relation \eqref{scale2vb} with $r=\vep^{-1}$ to see that 
\begin{align}\label{fbound2}
f_{\vep,t}(\lambda')&\le E^B_{\mu_\vep}\Bigl(\exp\Bigl(-\int_0^t \tilde V^{\lambda'\vep,1}(\vep^{-2}s,\vep^{-1}B_s)\vep^{-2}\,ds\Bigr)\Bigr)\cr
&=E^B_{\mu_1}\Bigl(\exp\Bigr(-\int_0^{t\vep^{-2}}\tilde V^{\lambda'\vep,1}(u,B_u)\,du\Bigr)\Bigr)\cr
&=E^B_{\mu_1}\Bigr(\exp\Bigr(-\int_0^1\tilde V^{\lambda'\vep,1}(u,B_u)\,du\Bigr)\psi(B_1)\Bigr),
\end{align}
where
\[\psi(x)=E^B_x\Bigl(\exp\Bigl(-\int_0^{t\vep^{-2}-1}\tilde V^{\lambda'\vep,1}(u+1,B_u)\,du \Bigr)\Bigr),\]
and we have used $\vep^{-2}t\ge1$.  So if $Y_s$ is the Ornstein-Uhlenbeck process $B(e^s-1)e^{-s/2}$, $T=\log(t\vep^{-2})$ and $b=\lambda'\vep$, then, as in the proof of Lemma~\ref{LTequation},
\begin{equation*}
\psi(x)=E^Y_x\Bigl(\exp\Bigl(-\int_0^TH^b(e^s,Y_s)\,ds\Bigr)\Bigr)=E_x^Y\Bigl(Z^b_T\exp\Bigl(-\int_0^TF(Y_s)\,ds\Bigr)\Bigr).
\end{equation*}
By Lemma~\ref{Zbnd2} and then \eqref{extbnd},
\begin{equation}\label{psibound1}\psi(x)\le C_{\ref{Zbnd2}}(b\wedge1)^{-20}P^Y_x(\rho>T)\le  C_{\ref{Zbnd2}}(b\wedge1)^{-20}c_\delta e^{\delta x^2}\vep^{2\lambda_0}t^{-\lambda_0}.
\end{equation}
Take $0<\delta<1/4$ and use the above in \eqref{fbound2} and then H\"older's inequality with $q=(4\delta)^{-1}$ and $p=(1-4\delta)^{-1}$, to get (with $b=\lambda'\vep$)
\[f_{\vep,t}(\lambda')\le c_\delta (b\wedge 1)^{-20} \vep^{2\lambda_0}t^{-\lambda_0}E^B_{\mu_1}\Bigl(\exp\Bigl(-\int_0^1\tilde V^{b,1}(u,B_u)\,du\Bigr)\Bigr)^{1-4\delta}.\]
So recalling \eqref{fbound} and using Lemma~\ref{smalltbnd}, we have
\begin{align*}
&P^X_{X_0}(0<X_t([x,x+\vep])\le \vep^2M)\cr
&\le c(t)X_0(1)\vep^{2\lambda_0}c'_\delta\int_{(\vep M)^{-1}}^\infty (\vep\lambda'\wedge 1)^{-20}\cr
&\phantom{\le c(t)X_0(1)\vep^{2\lambda_0}c'_\delta\int_{(\vep M)^{-1}}^\infty }\times\Bigl[E^B_{\mu_1}\Bigl(\exp\Bigl(-\int_0^1\tilde V^{\lambda'\vep,1}(u,B_u)\,du\Bigr)\Bigr) \Bigr]^{1-4\delta}d\lambda'\cr
&\le c(t)X_0(1)\vep^{2\lambda_0-1}c'_\delta\Bigl[\int_{M^{-1}}^1w^{-20}\,dw\cr
&\phantom{\le c(t)X_0(1)\vep^{2\lambda_0}c'_\delta\int_{(\vep M)^{-1}}^\infty }\times\int_1^\infty \Bigl[E^B_{\mu_1}\Bigl(\exp\Bigl(-\int_0^1\tilde V^{w,1}(u,B_u)\,du\Bigr)\Bigr) \Bigr]^{1-4\delta}\,dw\Bigr]\cr
&\le  c(t)X_0(1)\vep^{2\lambda_0-1}c'_\delta\Bigl[M^{19}+\int_1^\infty C_{~\ref{smalltbnd}}(\delta_0) w^{(-((1/2)+\lambda_0)+\delta_0)(1-4\delta)}\,dw\Bigr]\cr
&\le c(t)X_0(1)\vep^{2\lambda_0-1}M^{19},
\end{align*}
by choosing $\delta$ and $\delta_0$ sufficiently small since $\lambda_0>1/2$.  This gives the required result for $t=1$. By scaling (see, e.g., Exercise II.5.5 in \cite{per02}) it follows for general $t>0$.\qed

\medskip
\noindent {\bf Proof of Lemma~\ref{Zbnd2}.} Lemma~\ref{Fprop}(c) and \eqref{Vbound} show that $F,H^b\le 2$ and so by \eqref{LLEc} and \eqref{LLE},
\begin{align}\label{incbound1}
F(x)-H^b(u,x)&\le e^2\Bigl[e^{-H^b(u,x)}-e^{-F(x)}\Bigr]\cr
&\le e^2\Bigl[ E^X_{\delta_x}\Bigl(\exp\Bigl(-bu\int_0^{u^{-1/2}} X(1,y)\,dy\Bigr)-1(X(1,0)=0)\Bigr)\cr
&\le e^2\Bigl[E^X_{\delta_x}\Bigl(1(0<X(1,0)<u^{-1/8})\Bigr)\cr
&+E^X_{\delta_x}\Bigl(1(X(1,0)\ge u^{-1/8})\exp\Bigl(-bu^{1/2}\frac{\int_0^{u^{-1/2} }X(1,y)dy}{u^{-1/2}}\Bigl)\Bigr)\Bigr]\cr
&\equiv e^2(T_1+T_2).
\end{align}
Theorem~\ref{SBMdensity}(a) implies that 
\begin{equation}\label{T1bound8}
T_1\le Cu^{-(2\lambda_0-1)/8}.
\end{equation}
We apply the modulus of continuity in Proposition~\ref{classmod} with $\xi=3/8$ and $K=2$ and so
set $\rho_0(\omega)=\rho(2,3/8,\omega)$.  Choose $u_0$ so that 
\begin{equation}\label{u0def}
u^{-((\xi/2)-(1/8))}\le 1/2\text{ for all }u\ge u_0.
\end{equation}
If $u\ge u_0$, $\rho_0(\omega)\ge u^{-1/2}$ and $y\in[0,u^{-1/2}]$, then on $\{X(1,0)\ge u^{-1/8}\}$,
\[X(1,y)\ge X(1,0)-|y|^\xi\ge u^{-1/8}-u^{-\xi/2}\ge u^{-1/8}/2\text{ (by \eqref{u0def})}.\]
Therefore for $u\ge u_0$,
\begin{align}\label{T2bound8}
T_2&\le P^X_{X_0}(\rho_0<u^{-1/2})+\exp(-\frac{b}{2}u^{(1/2)-(1/8)})\cr
&\le Cu^{-\delta_{\ref{classmod}}/2}+\exp\Bigl(-\frac{b}{2}u^{3/8}\Bigr).
\end{align}
Combining \eqref{T1bound8} and \eqref{T2bound8} we have
\[F(x)-H^b(u,x)\le C(u^{-(2\lambda_0-1)/8}+u^{-\delta_{\ref{classmod}}/2})+e^2\exp(-\frac{b}{2}u^{3/8}),\]
first for $u\ge u_0$, and then for all $u\ge 1$ by increasing the constant $C$. 
Therefore,
\begin{align*}Z^b_T&\le C\exp\Bigl(\int_0^T e^2\exp\Bigl(-\frac{b}{2}e^{\frac{3s}{8}}\Bigr)\,ds\Bigr)\cr
&=C\exp\Bigl(\frac{8e^2}{3}\int _{b/2}^{(b/2)e^{3T/8}} e^{-w}w^{-1}\,dw\Bigr)\cr
&\le C\exp\Bigl(\frac{8e^2}{3}\int_{(b/2)\wedge 1}^1w^{-1}dw\Bigr)\le C[b\wedge 1]^{-20}.\qed
\end{align*}

\noindent{\bf Proof of Lemma~\ref{smalltbnd}.} Set $\gamma=\frac{1}{2}+\lambda_0\in(1,\frac{3}{2})$ and $\beta=\frac{1}{2}-\frac{\gamma}{4} \in(\frac{1}{8},\frac{1}{4})$.  For $\lambda\ge 1$ use an obvious symmetry to see that 
\begin{align}\label{psiexp}
\psi(\lambda)&\equiv E^B_{\mu_1}\Bigl(\exp\Bigl(-\int_0^1 \tilde V^{\lambda,1}(s,B_s)\,ds\Bigr)\Bigr)\cr
&=E^B_{\mu_1}\Bigl(1(B_0\in[\lambda^{-\beta},1-\lambda^{-\beta}])E^B_{B_0}\Bigl(\exp\Bigl(-\int_0^1 \tilde V^{\lambda,1}(s,B_s)\,ds\Bigr)\Bigr)\cr
&\quad+2E^B_{\mu_1}\Bigr(1(B_0\in[0,\lambda^{-\beta}])E^B_{B_0}\Bigl(\exp-\int_0^1 \tilde V^{\lambda,1}(s,B_s)\,ds\Bigr)\Bigr)\cr
&\equiv T_1+2T_2.
\end{align}

By Feynman-Kac we have 
\begin{equation}\label{FK}\tilde V^{\lambda,1}(t,x)=E^B_x\Bigl(\lambda1_{[0,1]}(B_t)\exp\Bigl(-\int_0^t\frac{\tilde V^{\lambda,1}(t-s,B_s)}{2}\,ds\Bigr)\Bigr).\end{equation}
If $\tilde V^{\lambda,1}_0$ was non-negative continuous, this would follow from \cite{ks91} (p. 268) and for general non-negative Borel initial conditions it follows by taking bounded pointwise limits (use \eqref{LLEc}).
Now bound $\tilde V^{\lambda,1}(u,x)$ above by $V(\lambda)(u,x)=\frac{2\lambda}{2+\lambda u}$ and use this bound in the exponent in \eqref{FK} to conclude that
\begin{align}\label{V1b}
\tilde V^{\lambda,1}(t,x)&\ge \lambda P^B_x(B_t\in [0,1])\exp\Bigl(-\int_0^t\frac{2\lambda}{2(2+\lambda s)}ds\Bigr)\cr
&=\frac{2\lambda}{2+\lambda t}P^B_x(B_t\in[0,1]).
\end{align}
If $x\in [\lambda^{-\beta}/2,1-(\lambda^{-\beta}/2)]$ and $t\le \lambda^{-2\beta(1+\vep)}$, for some $\vep>0$, then by \eqref{V1b},
\begin{align*}
\tilde V^{\lambda,1}(t,x)\ge& \frac{2\lambda}{2+\lambda t}(1-P^B_x(B_t\notin[0,1]))\cr
\ge& \frac{\lambda}{1+(\lambda t/2)}(1-\eta_\lambda),
\end{align*}
where $\eta_\lambda=\exp(-\lambda^{2\beta\vep/8)}$.
Therefore
\begin{align}\label{T1bound7}
T_1&\le \int_{\lambda^{-\beta}}^{1-\lambda^{-\beta}} E^B_x\Bigl(\exp\Bigl(-2\int_0^{\lambda^{-2\beta(1+\vep)}}\frac{(1-\eta_\lambda)\lambda}{(1+(\lambda s/2))2}ds\Bigr)\cr
&\phantom{\le \int_{\lambda^{-\beta}}^{1-\lambda^{-\beta}} E^0_x\Bigl(}\times1(\sup_{s\le\lambda^{-2\beta(1+\vep)}}|B_s-x|\le \frac{1}{2}\lambda^{-\beta})\Bigr)dx\cr
&\qquad +P_0^B\Bigl(\sup_{s\le \lambda^{-2\beta(1+\vep)}}|B_s|>\frac{1}{2}\lambda^{-\beta}\Bigr)\cr
&\le \Bigl[1+\frac{\lambda^{1-2\beta(1+\vep)}}{2}\Bigl]^{-2(1-\eta_\lambda)}+C\exp(-\lambda^{2\beta\vep}/8)\cr
&\le C_\vep\lambda^{-(1-2\beta(1+\vep))2(1-\eta_\lambda)}\le C_\vep\lambda^{-2(1-2\beta(1+2\vep))},
\end{align}
where the last inequality holds first for $\lambda\ge \lambda(\vep)$, and then for all $\lambda\ge 1$ by increasing $C_\vep$. 

For $T_2$ we use the scaling relation \eqref{scale2vb} with $r=\lambda^\beta$ to see that
\[\tilde V^{\lambda,1}(s,x)=\lambda^{2\beta}\tilde V^{\lambda^{1-2\beta},\lambda^\beta}(\lambda^{2\beta}s,\lambda^\beta x),\]
and so
\begin{align}\label{t2exp}T_2&=\lambda^{-\beta}E^B_{\mu_{\lambda^{-\beta}}}\Bigl(\exp\Bigl(-\int_0^1\tilde V^{\lambda^{1-2\beta},\lambda^{\beta}}(\lambda^{2\beta}s,\lambda^\beta B_s)\lambda^{2\beta}\,ds\Bigr)\Bigr)\cr
&=\lambda^{-\beta}E^B_{\mu_1}\Bigl(\exp\Bigl(-\int_0^{\lambda^{2\beta}} \tilde V^{\lambda^{1-2\beta},\lambda^\beta}(u,B_u)\,du\Bigr)\Bigr)\cr
&\le\lambda^{-\beta}E^B_{\mu_1}\Bigl(\exp\Bigl(-\int_0^1\tilde V^{\lambda^{1-2\beta},1}(u,B_u)\,du\Bigr)\hat\psi(B_1)\Bigr),
\end{align}
where
\[\hat\psi(x)=E^B_x\Bigl(\exp\Bigl(-\int_0^{\lambda^{2\beta}-1}\tilde V^{\lambda^{1-2\beta},1}(s+1,B_s)\,ds\Bigr)\Bigr).\]
By \eqref{psibound1} with $b=\lambda^{1-2\beta}$ and $\lambda^{2\beta}$ in place of $t\vep^{-2}$, for any $\delta>0$, 
\[\hat\psi(x)\le C_\delta e^{\delta x^2}\lambda^{-2\beta\lambda_0}.\]
Use this in \eqref{t2exp}, and then for any $p>1$ choose $\delta>0$ so that $1-2\delta>p^{-1}$, and apply H\"older's inequality to see that 
\begin{equation}\label{T2bound7}
T_2\le c_p\lambda^{-\beta(1+2\lambda_0)}E^B_{\mu_1}\Bigl(\exp\Bigl(-\int_0^1\tilde V^{\lambda^{1-2\beta,1}}(u,B_u)\,dy\Bigr)\Bigr)^{1/p}.
\end{equation}

Combine \eqref{psiexp}, \eqref{T1bound7} and \eqref{T2bound7} to conclude that for any $\vep>0$ and $p>1$ there are constants $C_\vep$ and $c_p$ so that 
\begin{align}\label{iterate}
\psi(\lambda)&\le C_\vep\lambda^{-2(1-2\beta(1+2\vep))}+c_p\lambda^{-\beta(1+2\lambda_0)}\psi(\lambda^{1-2\beta})^{1/p}\cr
&\le C_\vep\lambda^{-\gamma+2\vep}+c_p\lambda^{-2\beta\gamma}\psi(\lambda^{\gamma/2})^{1/p},
\end{align}
where the last line uses $\gamma>1$. Now fix $r\in(1,\gamma)$, choose $\vep>0$ so that 
\[\gamma-2\vep>r,\]
and then $p>1$ so that 
\[
2\beta\gamma+\frac{\gamma r}{2p}>r.
\]
The latter is easily seen to be possible by the choice of $r$ and a bit of arithmetic. With these choices of $\vep$ and $p$ we then may choose $\lambda'=\lambda'(r)$ sufficiently large so that 
(by our choice of $\vep$) for $C_\vep$ and $c_p$ as in \eqref{iterate}
\[
\lambda^{-\gamma+2\vep}\le\frac{1}{2C_\vep}\lambda^{-r}\text{ for }\lambda\ge \lambda'
\]
and (by our choice of $p$),
\[
\lambda^{-2\beta\gamma}\lambda^{-\gamma r/(2p)}\le\frac{1}{2c_p}\lambda^{-r}\text{ for }\lambda\ge \lambda'.
\]
Using the two bounds above we see that \eqref{iterate} becomes
\begin{equation}\label{iterate2}
\psi(\lambda)\le \frac{\lambda^{-r}}{2}+\frac{\lambda^{-r}}{2}\lambda^{\gamma r/(2p)}\psi(\lambda^{\gamma/2})^{1/p}\quad\text{for }\lambda\ge \lambda'.
\end{equation}
Let $N=N(r)$ be the minimal natural number such that $2^{(2/\gamma)^N}\ge\lambda'$. Now choose $C_0=C_0(r)\ge 1$ so that for all $1\le \lambda\le 2^{(2/\gamma)^N}$,
\begin{equation}\label{psiBND}\psi(\lambda)\le C_0\lambda^{-r}.
\end{equation}
We now prove by induction on $n\ge N$ that \eqref{psiBND} holds for $1\le\lambda\le2^{(2/\gamma)^n}$.  It holds for $n=N$ by our choice of $C_0$, so assume it holds for $1\le\lambda\le2^{(2/\gamma)^n}$ ($n\ge N$), and let $\lambda\in[2^{(2/\gamma)^n},2^{(2/\gamma)^{n+1}}]$. Then $\lambda\ge \lambda'$ and
\begin{equation}\label{indhyp} \lambda^{\gamma/2}\le 2^{(2/\gamma)^n}.
\end{equation}
Therefore by \eqref{iterate2}, \label{indhyp} and our induction hypothesis,
\begin{align*}\psi(\lambda)&\le \frac{\lambda^{-r}}{2}+\frac{\lambda^{-r}}{2}\lambda^{\gamma r/(2p)}C_0^{1/p}\lambda^{-\gamma r/(2p)}\cr
&\le C_0\lambda^{-r}.\end{align*}
This completes the induction and so for any $r<\frac{1}{2}+\lambda_0$, $\psi(\lambda)\le C_0(r)\lambda^{-r}$ for all $\lambda\ge 1$.\qed

\appendix
\section{Eigenfunction Expansions--Proof of Theorem~\ref{fens}}\label{efunctionexpa}
\setcounter{equation}{0}
\begin{proof} Define an isometry $T:L^2(m)\to L^2(dx)$ by $$Tf(x)=e^{-x^2/4}(2\pi)^{-1/4}f(x).$$ For $\phi$ as above set $q^\phi(x)=\frac{x^2}{8}-\frac{1}{4}+\phi$, and for $g\in C^2\cap L^2(dx)$, let $A^\phi g=T\circ(-L^\phi)\circ T^{-1}\in L^2(dx)$. Some calculus gives $A^\phi g=-\frac{1}{2}g''+q^\phi g$ and so
classical Sturm-Liouville theory as in E.g. 2 in Section 9.5 of \cite{CL55} 
shows that $A^\phi$ has a c.o.n.s. of eigenfunctions, $\{g_n:n\in\Z_+\}$ in $L^2(dx)\cap C^2$ with eigenvalues $\lambda_n\ge 0$ satisfying $\lambda_n\uparrow\infty$.  Clearly $\psi_n=T^{-1}\phi_n$ defines a c.o.n.s. in $L^2(m)$ of $C^2$ eigenfunctions for $L^\phi$, with corresponding eigenvalues $-\lambda_n\le 0$.  

\begin{prop} \label{mercer}Let $A(x,y)$ be a continuous symmetric kernel in $L^2(m\times m)$ with non-negative eigenvalues $\{\mu_n:n\in\Z_+\}$ and corresponding continuous eigenfunctions $\{\psi_n\}$.  Then 
$$A(x,y)=\sum_0^\infty \mu_n\psi_n(x)\psi_n(y),$$
where the convergence is in $L^2(m\times m)$ and absolutely uniformly on compacts. 
\end{prop}
\begin{proof} This follows as in the proof of Mercer's theorem on p. 245 of \cite{RN}.  One needs
to address the lack of compactness here but this is routine. 
\end{proof}

Recall (see e.g. Theorem 3.1.9 and Corollary 3.1.10 of \cite{kni81}) that Range$(R_\lambda^\phi)=D(L^\phi)=D(L)$ and
\begin{equation}\label{res}
(\lambda-L^\phi)R_\lambda^\phi f=f\ \forall f\in L^2(m);\quad R^\phi_\lambda(\lambda-L^\phi)f=f\ \forall f\in D(L).
\end{equation}
The latter clearly implies that $R^\phi_\lambda \psi_n=(\lambda+\lambda_n)^{-1}\psi_n$. If we set $\phi=0$ first, where it is well-known that $\lambda^0_n=n/2$ and $\psi^0_n$ is a multiple of the $n$th Hermite polynomial (see E.g. 2 in Section 9.5 of \cite{CL55}).  The transition density of the corresponding Ornstein-Uhlenbeck process with respect to $m$ is 
\begin{eqnarray}\label{q0}
q^0(t,x,y)&=&(1-e^{-t})^{-1/2}\exp\Bigl\{\frac{-x^2-y^2+2xye^{t/2}}{2(e^t-1)}\Bigr\},\ t>0\\
\nonumber&\equiv&(1-e^{-t})^{-1/2}\exp\Bigl\{h(t,x,y)\Bigr\},
\end{eqnarray}
and so its resolvent $R_\lambda$ is a self-adjoint Hilbert-Schmidt integral operator (since $\sum_n(\lambda+\lambda_n)^{-2}<\infty$) with kernel $G_\lambda(x,y)$ (with respect to $m$), where
\begin{equation}\label{Glamseries} G_\lambda(x,y)=\int_0^\infty q^0(t,x,y)e^{-\lambda t}\,dt=\sum_{n=0}^\infty (\lambda+\lambda^0_n)^{-1}\psi_n^0(x)\psi_n^0(y).
\end{equation}
Here the first expression readily shows $G_\lambda$ is in $L^2(m\times m)$ (use $\Vert q^0(t)\Vert_\infty\le c(1\vee t^{-1/2})$), continuous, and strictly positive, and so by Proposition~\ref{mercer} the above series converges in $L^2(m\times m)$
and absolutely uniformly for $x,y$ in compacts. One final bound we will need is
\begin{equation}\label{Glambound}
\sup_y G_\lambda(x,y)\le c_G(\lambda)e^{x^2/2}.
\end{equation}
This follows easily from the first equality in \eqref{Glamseries} and a bit of calculus which give
\begin{eqnarray*}G_\lambda(x,y)&\le& \int_0^\infty e^{-\lambda t}(1- e^{-t})^{-1/2}\sup_y \exp(h(t,x,y))\,dt\\
&=& \int_0^\infty  e^{-\lambda t}(1- e^{-t})^{-1/2}\exp(h(t,x,xe^{t/2}))\,dt.
\end{eqnarray*}
A short calculation now gives \eqref{Glambound}.

Clearly $L=1-(R_1)^{-1}$ is self adjoint and hence so are $L^\phi=L-\phi$ and $R_\lambda^\phi=(\lambda-L^\phi)^{-1}$.  Moreover $R^\phi_\lambda$ has a symmetric kernel in $L^2(m\times m)$ satisfying $G^\phi_\lambda(x,y)\le G_\lambda(x,y)$, and so is also a non-negative definite Hilbert-Schmidt integral operator.

\begin{lemma} \label{niceG} There is a jointly continuous symmetric version of $G^\phi_\lambda$ satisfying $0\le G^\phi_\lambda\le G_\lambda$ pointwise.

\end{lemma}
\begin{proof}
 It is easy to check using the Feynmann-Kac representation of $R_\lambda^\phi$ that 
\begin{equation*}
R_\lambda-R^\phi_\lambda=R_\lambda^\phi(\phi R_\lambda).
\end{equation*}
This implies that 
\begin{equation}\label{deltaG}
G_\lambda(x,y)-G_\lambda^\phi(x,y)=\int G_\lambda^\phi(x,z)\phi(z) G_\lambda(z,y)\,dm(z).
\end{equation}
If $\Delta_\lambda(x,y)$ be defined by the right-hand side of \eqref{deltaG}, then as $G^\phi_\lambda$ is only defined up to $m\times m$-null sets,  $x\to \Delta_\lambda(x,\cdot)$ is uniquely determined up to $m$-null sets of $x$.  
We also may assume by truncation that $0\le G^\phi_\lambda(x,y)\le G_\lambda(x,y)$ for all $x,y$.  We claim that for fixed $x$
\begin{eqnarray}\label{Deltacont}&&\forall K, \varepsilon>0,\ \exists \delta>0 \text{ s.t. }\forall x,y,y'\in[-K,K],\\
\nonumber&& |y-y'|<\delta\Rightarrow |\Delta_\lambda(x,y)-\Delta_\lambda(x,y')|<\varepsilon.
\end{eqnarray}
Let $S_N(x,y)=\sum_0^N(\lambda+\lambda_N)^{-1}\psi^0_N(x)\psi^0_N(y)$ and 
$$d_N(y,y')=\sum_0^N(\lambda+\lambda_N)^{-2}(\psi_N^0(y)-\psi^0_N(y'))^2.$$
Then by Cauchy-Schwarz and Fatou's Lemma,
\begin{align}
\nonumber&|\Delta_\lambda(x,y)-\Delta_\lambda(x,y')|^2\\
\nonumber&\quad\le c\int G_\lambda(x,z)^2dm(z)\int (G_\lambda(z,y)-G_\lambda(z,y'))^2dm(z)\\
\nonumber&\quad\le \liminf_{N\to\infty}c\int S_N(x,z)^2dm(z)\int(S_N(z,y)-S_N(z,y'))^2dm(z)\\
\nonumber&\quad=\Bigl(\lim_{N\to\infty}c\sum_{n=0}^N (\lambda+\lambda_n)^{-2}\psi^0_n(x)^2\Bigr)\Bigl(\lim_{N\to\infty}\sum_{n=0}^N(\lambda+\lambda_N)^{-2}(\psi^0_n(y)-\psi^0_n(y'))^2\Bigr)\\
\nonumber&\quad\le\Bigl(c\lambda^{-1}\sum_{n=0}^\infty(\lambda+\lambda_n)^{-1}\psi_n^0(x)^2\Bigr)\\
\nonumber&\qquad\qquad\qquad\times\Bigl(d_N(y,y')+2\sum_{n>N}(\lambda+\lambda_n)^{-2}(\psi^0_n(y)^2+\psi_n^0(y')^2)\Bigr)\\
\label{deltaincr}&\quad\le c\lambda^{-1}G_\lambda(x,x)\Bigl(d_N(y,y')+\frac{2}{\lambda+\lambda_N}(G_\lambda(y,y)+G_\lambda(y',y'))\Bigr)\\
\nonumber&\quad\le c(\lambda,K)\Bigl(d_N(y,y')+\lambda_N^{-1}\Bigr),
\end{align}
where the last line holds for $x,y,y'\in[-K,K]$.  By first taking $N$ large and then $\delta$ small enough we get \eqref{Deltacont}. 

Define $\tilde G_\lambda^\phi(x,y)=G_\lambda(x,y)-\Delta_\lambda(x,y)(\le G_\lambda(x,y))$. Note that $x\to\tilde G_\lambda^\phi(x,\cdot)$ is defined up to $m$-null sets as a continuous function-valued map and that $\tilde G^\psi_\lambda=G_\lambda^\phi\ m\times m$-a.e. Pick a version of $\tilde G^\phi_\lambda\in[0, G_\lambda]$ and define
\begin{eqnarray}
\nonumber \hat G_\lambda^\phi(x,y)&\equiv&G_\lambda(x,y)-\int\tilde G^\phi_\lambda(z,x)\phi(z)G_\lambda(z,y)\,dm(z)\\
\label{hatGdef}&\equiv&G_\lambda(x,y)-\hat\Delta_\lambda(x,y).
\end{eqnarray}
We claim that $\hat G^\phi_\lambda$ is a jointly continuous version of $G_\lambda^\phi$. We have $\tilde G_\lambda^\phi(x,y)=\tilde G^\phi_\lambda(y,x)$ $m$-a.e. and so for $m$-a.a. $(x,y)$,
$$\hat\Delta_\lambda(x,y)=\int\tilde G^\phi_\lambda(x,z)\phi(z)G_\lambda(z,y)\,dm(z)=\Delta_\lambda(x,y),
$$
which implies (by \eqref{deltaG}) $\hat G^\phi_\lambda(x,y)=G_\lambda^\phi(x,y)$ $m\times m$-a.e.

Just as in the proof of \eqref{Deltacont} (replace $G_\lambda^\phi(x,z)$ by $\tilde G^\phi_\lambda(z,x)$), we have 
\begin{eqnarray}\label{hatDeltacont}&&\forall K, \varepsilon>0,\ \exists \delta>0 \text{ s.t. }\forall x,y,y'\in[-K,K],\\
\nonumber&& |y-y'|<\delta\Rightarrow |\hat\Delta_\lambda(x,y)-\hat\Delta_\lambda(x,y')|<\varepsilon.
\end{eqnarray}
Fix $y\in\R$.  Then 
\begin{align}\label{hatdeltincr}
 |\hat\Delta_\lambda&(x,y)-\hat\Delta_\lambda(x',y)|^2\\
\nonumber&\le\int (\tilde G^\phi_\lambda(z,x)-\tilde G^\phi_\lambda(z,x'))^2\phi(z)^2\,dm(z)\,\int G_\lambda(z,y)^2\,dm(z)\\
\nonumber&\le c\int (\tilde G^\phi_\lambda(z,x)-\tilde G^\phi_\lambda(z,x'))^2\,dm(z)\,\lambda^{-1}G_\lambda(y,y)\ \ \text{(as in \eqref{deltaincr})}.
\end{align}
For each $z$, \eqref{Deltacont} implies that $\lim_{x'\to x}(\tilde G^\phi_\lambda(z,x)-\tilde G^\phi_\lambda(z,x'))^2=0$.  To apply Dominated Convergence we may use the bound in \eqref{Glambound} to see that ($x$ fixed)
\begin{align*}
\sup_{|x'-x|\le 1, z}&(\tilde G^\phi_\lambda(z,x)-\tilde G^\phi_\lambda(z,x'))^2\\
&\le 2\sup_zG_\lambda(x,z)^2+2\sup_{|x'-x|\le 1,z}G_\lambda(x',z)^2\\
&\le 2c_G(\lambda)^2(e^{x^2}+\sup_{|x'-x|\le 1}e^{(x')^2})=C(\lambda,x).
\end{align*}
This shows that the expression on the right-hand side of \eqref{hatdeltincr} approaches $0$ as $x'\to x$. This, together with \eqref{hatDeltacont} implies the joint continuity of $\hat\Delta_\lambda$, and so by \eqref{hatGdef}, the required joint continuity of $\hat G^\phi_\lambda$. The other properties of $\hat G^\phi_\lambda$ are then immediate.
\end{proof}

We return to the proof of Theorem~\ref{fens}(a). Recall we use $P_x$ for the law of the Ornstein-Uhlenbeck process starting at $x$. The fact that $f\ge 0$ and $m(f>0)>0$ implies $R^\phi_\lambda f(x)>0$ for all $x$ and the continuity of $G_\lambda^\phi$ imply $G^\lambda\ge 0$ and is strictly positive for $m\times m$ a.a. $(x,y)$.  It is then easy to show that $\psi_0$ cannot change sign and so may be taken to be non-negative and hence $\lambda_0$ must be simple  as in Section 23.1 of \cite{Lax02}. If for some $x_0$ $\psi(x_0)=0$, then we also have $\psi_0'(x_0)=0$ which would imply $\psi_0\equiv 0$ as it is the solution of nice second order ode.  Hence $\psi_0>0$ and the proof of (a) is complete.

(b),(c) Lemma~\ref{niceG} allows us to apply Proposition~\ref{mercer} to see that 
\begin{equation}\label{Glamconv} G^\phi_\lambda(x,y)=\sum_0^\infty\frac{1}{\lambda+\lambda_n}\psi_n(x)\psi_n(y)
\end{equation}
 in $L^2(m\times m)$ and absolutely uniformly on compacts.
A simple comparison of coefficients then shows that
\begin{equation}\label{qseries} q(t,x,y)\equiv\sum_{n=0}^\infty e^{-\lambda_n t}\psi_n(x)\psi_n(y) 
\end{equation}
 in $L^2(m\times m)$ for each $t>0$ and absolutely uniformly in $(t,x,y)\in[\vep,\infty)\times[\vep,\vep^{-1}]$ for all $\vep>0$,
where the limit $q$ is necessarily jointly continuous in $(t,x,y)$. Moreover 
we can integrate both sides of \eqref{qseries} with respect $e^{-\lambda t}f(y)\,dt\,dm(y)$, where $f$ is continuous on compact support (the absolute uniform convergence gives the integrability required for interchange of limits) to see that 
$$\int_0^\infty e^{-\lambda t}\int q(t,x,y) f(y)\,dm(y)\,dt=\int G_\lambda(x,y)f(y)\,dm(y)\text{ for all }x.$$
This shows that $q(t,x,y)$ is a jointly continuous version of the transition density of the killed Ornstein-Uhlenbeck process $Y^\phi$. To obtain \eqref{qbound1} one proceeds as in the proof of Theorem~1.1 in \cite{Uch80}. If $t\ge s^*(\delta)$, then by \eqref{qseries},
\begin{align}\label{qbound2}
e^{\lambda_0t}q(t,x,y)&\le \sum_0^\infty e^{-(\lambda_n-\lambda_0)t}|\psi_n(x)||\psi_n(y)|\cr
&\le \sum_0^\infty e^{-(\lambda_n-\lambda_0)s^*}|\psi_n(x)||\psi_n(y)|\cr
&\le e^{\lambda_0 s^*}\sqrt{q_{s^*}(x,x)}\sqrt{q_{s^*}(y,y)}\cr
&\le e^{\lambda_0 s^*}\sqrt{q^0_{s^*}(x,x)}\sqrt{q^0_{s^*}(y,y)}\cr
&\le c(\delta)\exp(\delta(x^2+y^2)),
\end{align}
where in the last line we use \eqref{q0} and the definition of $s^*$. 

(d) The $L^2$ convergence in \eqref{qseries} readily implies 
$$P_x(\rho_\phi>t)=\int q(t,x,y)\,dm(y)=\sum_0^\infty e^{-\lambda_n t}\psi_n(x)\int\psi_n(y)\,dm(y),$$
where the convergence in the series is in $L^2(m)$. It follows that 
\begin{equation}\label{rhoexp}P_x(\rho_\phi>t)=e^{-\lambda_0 t}[\psi_0(x)\theta+r(t,x)],
\end{equation}
where 
\begin{equation}\label{rdef} r(t,x)=e^{\lambda_0 t}\sum_{n=1}^\infty e^{-\lambda_nt}\psi_n(x)\int \psi_n(y)\,dm(y),
\end{equation}
where the series converges in $L^2(m)$ and we recall that $\theta=\int\psi_0\,dm$.  One now proceeds as in the derivation of \eqref{qbound2} to see that in fact the 
above convergence is absolute uniform for $x,y$ in compacts and that 
\begin{eqnarray}\label{rbound2}&&\text{ for any $\delta>0$ there is a $c_\delta$ so that for all }t\ge s^*(\delta),\\
\nonumber && |r(t,x)|\le \sum_{n=1}^\infty e^{-(\lambda_n-\lambda_0)t}|\psi_n(x)|\int|\psi_n(y)|\,dm(y)\le c_\delta e^{\delta x^2}e^{-(\lambda_1-\lambda_0)t}.
\end{eqnarray}
 We will return to \eqref{rbound} for $t<s^*(\delta)$ below.

As in the derivation of \eqref{qbound2} we also have
\begin{equation}\label{q0bound}\sum_{n=0}^\infty e^{-\lambda_n s^*}|\psi_n(x)|\,|\psi_n(y)|\le \Bigl(q(s^*,x,x)q(s^*,y,y)\Bigr)^{1/2}\le c_0(\delta)\exp\{\delta(x^2+y^2)\}.
\end{equation}
Therefore, for $T>s^*$ (the convergence in the series below is in $L^2(m)$ and absolutely uniformly for $x$ in compacts),
\begin{eqnarray*}
e^{\lambda_0 T}P_x(\rho_\phi>T)&=&e^{\lambda_0 T}\Bigl|\sum_{n=0}^\infty e^{-\lambda_n T}\psi_n(x)\int\psi_n(y)\,dm(y)\Bigr|\\
&\le& e^{\lambda_0s^*}\sum_{n=0}^\infty e^{-\lambda_n s^*}|\psi_n(x)|\int |\psi_n(y)|\,dm(y)\\
&\le&c_0(\delta)e^{\lambda_0 s^*}e^{\delta x^2}\int e^{\delta y^2}\,dm(y)\quad\text{by \eqref{q0bound})}\\
&\equiv&c_1(\delta)e^{\delta x^2}.
\end{eqnarray*}
Combine \eqref{rhoexp}, \eqref{rbound2} with the above to see that 
$$\theta\psi_0(x)=\lim_{T\to\infty} e^{\lambda_0 T}P_x(\rho_\phi>T)\le c_1(\delta) e^{\delta x^2}.$$
It remains to prove \eqref{rbound} for $0\le t\le s^*(\delta)$. By definition and the above, we have
\begin{align*}
|r(t,x)|&\le e^{\lambda_0 t}P_x(\rho_\phi>t)+\theta\psi_0(x)\cr
&\le (e^{\lambda_0s^*(\delta)}+c_1(\delta))e^{\delta x^2}\le c_\delta e^{-(\lambda_1-\lambda_0)t}e^{\delta x^2}.
\end{align*}

(e) Fix $x\in\R$.  If $0\le t_1<t_2<T$ and $\phi_i$ are bounded measurable functions, then
\begin{eqnarray}&&E_x(\prod_1^2\phi_i(Y_{t_i})1(\rho_\phi>T))/P_x(\rho_\phi>T)\nonumber\\
\nonumber&&\quad=E_x(\prod_1^2\phi_i(Y_{t_i})1(\rho_\phi>t_2)P_{Y_{t_2}}(\rho_\phi>T-t_2))/P_x(\rho_\phi>T)\\
\nonumber&&\quad=\Bigl[E_x(\prod_1^2\phi_i(Y_{t_i})1(\rho_\phi>t_2)\frac{\theta\psi_0(Y_{t_2})e^{-\lambda_0(T-t_2)}}{\theta\psi_0(x)e^{-\lambda_0 T}}\\
\label{Tdispl}&&\qquad+E_x(\prod_1^2\phi_i(Y_{t_i})1(\rho_\phi>t_2)\frac{e^{-\lambda_0(T-t_2)}r(T-t_2,Y_{t_2}))}{\theta\psi_0(x)e^{-\lambda_0 T}}\Bigr]\\
&&\qquad\quad\times\frac{\theta\psi_0(x)e^{-\lambda_0 T}}{\theta\psi_0(x)e^{-\lambda_0 T}+r(T,x)e^{-\lambda_0 T}}\nonumber\\
&&\quad\equiv[T_1+T_2]\times T_3,\nonumber
\end{eqnarray}
where \eqref{rhotail} is used in the second equality.  By using \eqref{rbound} we have 
\begin{eqnarray*}
|T_2|&\le&\frac{\Vert\phi_1\Vert_\infty\Vert\phi_2\Vert_\infty}{\theta\psi_0(x)}e^{\lambda_0 t_2}E_x(c_\delta e^{\delta Y^2_{t_2})}e^{-(\lambda_1-\lambda_0)(T-t_2)}\\
&\le&c(x)e^{\lambda_0t_2}e^{-(\lambda_1-\lambda_0)(T-t_2)}\to 0\text{ as }T\to\infty.
\end{eqnarray*}
By \eqref{rbound} we also have
\begin{eqnarray*}
|T_3-1|&\le&\frac{|r(T,x)|e^{-\lambda_0 T}}{\theta\psi_0(x)e^{-\lambda_0 T}-|r(T,x)|e^{-\lambda_0 T}}\\
&\le &\frac{c_\delta e^{\delta x^2}e^{-(\lambda_1-\lambda_0)T}}{\theta\psi_0(x)-c_\delta e^{\delta x^2}e^{-(\lambda_1-\lambda_0)T}}\to0\text{ as }T\to\infty.
\end{eqnarray*}
Use these last results in \eqref{Tdispl} to conclude that
\begin{eqnarray*}
&&\lim_{T\to\infty}E_x\Bigl(\prod_1^2\phi_i(Y_{t_i})|\rho_\phi>T\Bigr)\\
&&\quad=E_x\Bigl(\prod_1^2\phi_i(Y_{t_i})1(\rho_\phi>t_2)\psi_0(Y_{t_2})e^{\lambda_0t_2}\Bigr)/\psi_0(x)\\
&&\quad=\int\int q(t_1,x,x_1)q(t_2-t_1,x_1,x_2)\prod_1^2\phi_i(x_i)\frac{\psi_0(x_2)}{\psi_0(x)}e^{\lambda_0 t_1}e^{\lambda_0(t_2-t_1)}dm(x_1)dm(x_2)\\
&&\quad=\int\int\tilde q(t_1,x,x_1)\tilde q(t_2-t_1,x_1,x_2)\prod_1^2\phi_i(x_i)dm(x_1)dm(x_2).
\end{eqnarray*}
Similar reasoning gives the convergence of the $k$-dimensional distributions for all $k$.

It remains to establish tightness. If $0\le s<t\le t_0<T$ with $t-s\le 1$, then
\begin{equation}\label{tightbnd}E_x((Y_t-Y_s)^4|\rho_\phi>T)\le E_x((Y_t-Y_s)^4P_{Y_t}(\rho_\phi>T-t))/P_x(\rho_\phi>T).\end{equation}
It is now easy to use \eqref{rhotail}, \eqref{rbound} and \eqref{psi0bnd}, together with Cauchy-Schwarz to bound the right-hand side of \eqref{tightbnd} by $c'(x,t_0)(t-s)^2$ at least for $T>T_0(x,t_0)$.  This gives the required tightness and the proof of (e) is complete.
\end{proof}
\section{A Tauberian Theorem--Proof of Lemma~\ref{Taub}}\label{Tauberian}
\begin{proof} (a) $U(a)\le \int_0^a e^{1-(x/a)}dU(x)\le e\hat U(a^{-1})\le eC_2a^p$.

\noindent (b) Define $\beta$ so that $d_1=\frac{C_1}{2}\beta^{-p}$.  Then $\beta\ge 2p$ implies $e^{-x/2}x^p$ is decreasing on $[\beta,\infty)$ and so 
\begin{equation}\label{calcbnd}
\int_\beta^\infty e^{-y}y^p\,dy\le e^{-\beta/2}\beta^p\int_\beta^\infty e^{-y/2}\,dy=2e^{-\beta}\beta^p.
\end{equation}
If $\lambda\ge\beta(\ge \underline\lambda)$, 
an integration by parts gives
\begin{align*}
\hat U(\lambda)=\int_0^\infty \lambda e^{-\lambda x}U(x)\,dx&=\int_0^\infty e^{-y}U(y/\lambda)\,dy\cr
&\le U(\beta/\lambda)+\int_\beta^\infty e^{-y}U(y/\lambda)\,dy\cr
&\le U(\beta/\lambda)+eC_2\lambda^{-p}\int_\beta^\infty e^{-y}y^p\,dy,
\end{align*}
the last by (a). So rearrange the above and use \eqref{calcbnd} and \eqref{LTLbound} to see that 
\begin{equation}\label{ULB}
U(\beta/\lambda)\ge C_1\lambda^{-p}-eC_2\lambda^{-p}2e^{-\beta}\beta^p
\end{equation}
Now 
\begin{align*}
e^{-\beta}\beta^p\le \sup_x(e^{-x/2}x^p)e^{-\beta/2}&=e^{-p}(2p)^pe^{-\beta/2}\cr
&\le \frac{C_1}{4eC_2},
\end{align*}
where the last line follows from the definition of $\beta$.  Use this in \eqref{ULB}
to conclude $U(\beta/\lambda)\ge (C_1/2)\lambda^{-p}$ for all $\lambda\ge \beta$.  Now set $a=\beta/\lambda\le 1$, to see that $U(a)\ge \frac{C_1}{2}\beta^{-p}a^p$, which gives \eqref{ULbound}.  \eqref{ULboundb} now follows.
\end{proof}

\section{Improved Modulus of Continuity--Proof of Theorem~\ref{improvedmod}}\label{impmoda}

For $K,N\in\N$, let $T_K=\inf\{t\ge 0: \sup_x X(t,x)\ge K\}\wedge K$, and 
\begin{align*}Z^0(K,N)=&\{(t,x):t\le T_K,\,|x|\le K,\ \exists(\hat t,\hat x), \text{ s.t. }\hat t\le T_K,\cr
&\quad X(\hat t,\hat x)\le 2^{-N},\text{ and }d((t,x),(\hat t,\hat x))\le 2^{-N}\}.
\end{align*}
We let $C^{1,+}_0$ denote the space of continuously differentiable, compactly supported non-negative functions on the line. 
\begin{lemma}\label{smoothic} Assume $X_0\in C_0^{1,+}$ and $\xi\in(0,1)$. For any $K\in\N$, there is an $N_0=N_0(K,\omega)\in\N$ a.s. such that for all $N\ge N_0$ and $(t,x)\in Z^0(K,N)$, whenever $t'\le T_K$ and $d((t',x'),(t,x))\le 2^{-N}$, then $X(t',x')\le 2^{-N\xi}$. 
\end{lemma}
\begin{proof} This is proved as for Theorem 2.2 of \cite{mp11}, using the iteration
result Corollary~4.2 of \cite{mps06}.  The computations are easier than in the coloured noise setting of  \cite{mps06} as we now have spatial orthogonality.  The only difference is that in the above references one was looking at the difference of two solutions rather 
than a single solution and this means the lead term in 
\[X(t,x)=P_tX_0(x)+\int_0^t\int p_{t-s}(y,x)\sqrt{X(s,y)}dW(s,y),\]
(see e.g., (III.4.11) of \cite{per02} but now with $t_0=0$ as we have a nice initial density)
 does not cancel and must be handled separately. For this, our hypothesis on $X_0$ implies that $P_tX_0(x)$ is $d$-Lipschitz continuous in $(t,x)$. As a result the iteration in \cite{mps06} may be first applied to the above stochastic integral term and then transferred over to $X$. 
\end{proof}
We next consider  general initial conditions.  For $N\in\N$, set 
\[Z(N)=\{(t,x):t> 0,\ \exists\, (\hat t,\hat x)\text{ s.t. }X(\hat t,\hat x)\le 2^{-N},\ d((\hat t, \hat x),(t,x))\le 2^{-N}\}.\]
\begin{prop} \label{genie}Assume $X_0$ is a finite measure on the line and $\xi\in(0,1)$.  For any $K\in\N$ there is an $N_1(K,\omega)\in \N$ a.s. so that for all $N\ge N_1$ and for all $(t,x)\in Z(N)$ with $t\ge K^{-1}$, $d((t',x'),(t,x))\le 2^{-N}$ implies $X(t',x')\le 2^{-N\xi}$.  
\end{prop}
\begin{proof} 
Let $X_0'\in C_0^{1,+}$, not identically $0$, and let  $X_0$ be as in the Proposition.  Consider the decreasing sequence of events,
\begin{align*}
B_{K'}=&\{X:\exists\,N_0\in\N\text{ s.t. }\forall N\ge N_0,\ \forall (t,x)\in Z^0(K',N),\cr
&\qquad t'\le T_{K'},\ d((t',x'),(t,x))\le 2^{-N}\text{ implies }X(t',x')\le 2^{-N\xi}\}.
\end{align*}
For $K\in\N$ fixed, Lemma~\ref{smoothic} and the Markov property imply that 
\[\text{for }P^X_{X_0'}-\text{a.a. }\omega,\ \forall K'\in\N,\text{ on }\{\omega:T_{K'}>(2K)^{-1}\},\ P_{X_{1/2K}(\omega)}(B_{K'-1})=1.\]
Clearly for $P^X_{X_0'}$-a.a. $\omega$, for $K'$ sufficiently large, $T_{K'}(\omega)>(2K)^{-1}$.  It follows that 
\[\text{for }P^X_{X_0'}-\text{a.a. }\omega,\ P^X_{X_{1/2K}(\omega)}(\liminf_{K'} B_{K'})=1.\]
Under $P^X_{X_{1/2K}(\omega)}$, by taking $K'$ sufficiently large so that
\[\{(t,x):X(t,x)>0\}\subset [0,K'-1]\times[-K'+1,K'-1]\ \text{and }\sup_{(t,x)}X(t,x)<K',\]
we see that
\begin{align*}
\liminf_{K'}B_{K'}&\subset\{X:\exists N_0\in\N^{\ge 4K}\text{ s.t. for } N\ge N_0,\ (t,x)\in Z(N),\cr
&\qquad d((t',x'),(t,x))\le 2^{-N}\ \text{implies }X(t',x')\le 2^{-N}\}\cr
&\equiv B_\infty.
\end{align*}
So for $P^X_{X_0'}$-a.a. $\omega$, $P^X_{X_{1/2K}(\omega)}(B_\infty)=1$.  By the equivalence of $P^X_{X_0}(X_{1/2K}\in\cdot)$ and $P^X_{X_0'}(X_{1/2K}\in\cdot)$ (see Corollary~2.2 of \cite{EP91}), we have  for $P^X_{X_0}$-a.a. $\omega$, $P^X_{X_{1/2K}(\omega)}(B_\infty)=1$. 
Another application of the Markov property now gives the required result. For this last step note that $(t,x)\in Z(N)$, $t\ge 1/K$ and $N\ge 4K$ imply that if $(\hat t,\hat x)$ are as in the definition of $Z(N)$ and $d((t',x'),(t,x))\le 2^{-N}$, then $\hat t\ge K^{-1}-2^{-N/2}\ge (2K)^{-1}$ and similarly $t'\ge (2K)^{-1}$. \end{proof}

{\bf Proof of Theorem \ref{improvedmod}.} This now follows upon noting that $Z\subset Z(N)$ and applying the above Proposition and a trivial interpolation argument.{\flushright\qed}

\section*{Acknowledgement} We thank Roger Tribe for sharing his thoughts and unpublished notes on this problem with us. 

\def\cprime{$'$} \def\cprime{$'$} \def\cprime{$'$}

\end{document}